\documentclass[11pt]{article}

\usepackage{amsfonts, amsmath, amssymb, latexsym, eucal, amscd} 
\usepackage{cite}
\usepackage[dvips]{pict2e}
\usepackage[all]{xy}
\newtheorem{theorem}{Theorem}[section]
\newtheorem{lemma}[theorem]{Lemma}
\newtheorem{prop}[theorem]{Proposition}
\newtheorem{corollary}[theorem]{Corollary}
\newtheorem{exAux}[theorem]{Example}

\newtheorem{Def}[theorem]{Definition}
\newenvironment{defi}{\begin{Def} \rm}{\end{Def}}
\newtheorem{Note}[theorem]{Note}
\newenvironment{note}{\begin{Note} \rm}{\end{Note}} 
\newtheorem{Problem}[theorem]{Problem}

\newtheorem{Rem}[theorem]{Remark}

\newtheorem{Not}[theorem]{Notation}

\newtheorem{Conj}[theorem]{Conjecture}

\newtheorem{Ass}[theorem]{Assumption}

\newenvironment{proof}{\medskip\noindent{\bf Proof.\ }}{\qed\medskip}

\newcommand{\qed}{\hfill\mbox{$\Box$\qquad\qquad}}

\newcommand{\F}{{\mathbb F}}

\newcommand{\vphi}{\varphi}
\renewcommand{\th}{\theta}

\newif\ifDRAFT
\DRAFTfalse

\begin{document}

\title{Spin Leonard pairs and the zero diagonal space}
\author{Kazumasa Nomura and Paul Terwilliger}

\maketitle

\begin{abstract}
Let  $\F$ denote a field,
and let $V$ denote a vector space over $\F$ with finite positive dimension.
A Leonard pair on $V$ is an ordered pair of diagonalizable linear maps
$A : V \to V$ and
$A^* : V \to V$
that each act on an eigenbasis of the other one in an irreducible tridiagonal fashion.
A Leonard pair $A,A^*$ on $V$ is said to have spin whenever
there exist invertible linear maps $W : V \to V$ and
$W^* : V \to V$
such that
$W A = A W$
and
$W^* A^* = A^* W^*$
and
$W A^* W^{-1} = (W^*)^{-1} A W^*$.
Let $\{\th^*_i\}_{i=0}^d$ denote a standard ordering of the eigenvalues of $A^*$.
There is a related sequence of scalars $\{a_i\}_{i=0}^d$  called intersection numbers.
The Leonard pair $A,A^*$ is called self-dual whenever $\{\th^*_i\}_{i=0}^d$ is
a standard ordering of the eigenvalues of $A$.
We obtain the following results under the assumption that
$\F$ is algebraically closed and $d \geq 3$.
We show that a Leonard pair $A,A^*$ on $V$ has spin 
if and only if both
(i)
$A,A^*$ is self-dual;
(ii)
there exist scalars $f_0,f_1,f_2, f_3$ (not all zero) such that
$
 f_0 + f_1 \th^*_i + f_2 a_i + f_3 a_i \th^*_i = 0$
for $0 \leq i \leq d$.
We also classify the Leonard pairs $A,A^*$ on $V$ that satisfy (ii)
without assuming (i).
To do this we bring in the following maps.
For $0 \leq i \leq d$ let $E^*_i : V \to V$ denote the projection onto
the $\th^*_i$-eigenspace of $A^*$.
Let ${\mathcal Z}(A,A^*)$ denote the set of elements $X$ 
in $\text{Span}\{I, A^*, A, A A^*\}$
such that $E^*_i X E^*_i= 0$ for $0 \leq i \leq d$.
We call
${\mathcal Z}(A,A^*)$ 
the zero diagonal space of $A,A^*$.
As we will see,
${\mathcal Z}(A,A^*) \neq 0$ 
if and only if the above condition (ii) holds.
As we investigate the case   ${\mathcal Z}(A,A^*) \neq 0$ in detail,
we break the problem into 13 cases called types;
these are the $q$-Racah type and its relatives.
For each type we give a necessary and sufficient condition 
for ${\mathcal Z}(A,A^*) \not=0$. 
For each type we give an explicit basis for  ${\mathcal Z}(A,A^*)$.
\end{abstract}

\section{Introduction}
\label{sec:intro}
\ifDRAFT {\rm sec:intro}. \fi

The notion of a Leonard pair was introduced by the second author in \cite{T:Leonard}.
A Leonard pair is described as follows.
Let $\F$ denote a field,
and let $V$ denote a vector space over $\F$ with finite positive dimension.
A Leonard pair on $V$ is an ordered pair of diagonalizable linear maps 
$A : V \to V$ and $A^* : V \to V$  
that each act on an eigenbasis of the other one
in an irreducible tridiagonal fashion \cite[Definition 1.1]{T:Leonard}.
The Leonard pairs are classified in  \cite[Theorem 1.9]{T:Leonard},
and the solutions correspond to the  orthogonal polynomials
in the terminating branch of the Askey scheme  \cite[Section 5]{T:parray}, \cite{T:survey}.

In the present paper we  consider a type of Leonard pair, said to have spin.
The spin condition is described as follows.
According to \cite[Definition 1.2]{Cur},
a Leonard pair $A,A^*$ on $V$ has spin
whenever there exist invertible linear maps $W: V \to V$ and $W^* : V \to V$
such that
$W A = A W$ and
$W^* A^* = A^* W^*$ and
$W A^* W^{-1} = (W^*)^{-1} A W^*$.
The spin Leonard pairs are classified in \cite[Theorem 1.13]{Cur}. 
For more information about spin Leonard pairs, see \cite{CaW, Cur, Cur2, NT:spin, NT:spin2}.

As we describe the features of a spin Leonard pair,
we will use the the following parameters.
Let $A,A^*$ denote a Leonard pair on $V$.
Let $\{\th^*_i\}_{i=0}^d$
denote a standard ordering of the eigenvalues of $A^*$ (see the paragraph  above Definition \ref{def:LS}).
Associated with $\{\th^*_i\}_{i=0}^d$ is a sequence of scalars $\{a_i\}_{i=0}^d$ called
intersection numbers (see Definition \ref{def:ai}).
The Leonard pair $A,A^*$ is said to be self-dual whenever $\{\th^*_i\}_{i=0}^d$ is
a standard ordering of the eigenvalues of $A$.
According to \cite[Theorem 1.9]{Cur} and  \cite[Lemma 9.2]{Cur2},
if $A,A^*$ has spin then
\begin{itemize}
\item[(i)]
 $A,A^*$ is self-dual.
\end{itemize}
By \cite[Lemmas 1.7--1.11, Theorem 1.13]{Cur},
if $A,A^*$ has spin then
\begin{itemize}
\item[(ii)]
there exist scalars $f_0,f_1,f_2,f_3$ (not all zero) such that
\[
 f_0 + f_1 \th^*_i + f_2 a_i + f_3 a_i \th^*_i = 0
 \qquad \qquad (0 \leq i \leq d).        
\]
\end{itemize}

In the present paper,
we obtain the following results under the assumption that $\F$ is algebraically closed
and $d \geq 3$.
We show that a Leonard pair  $A,A^*$ on $V$ has spin if and only if both
(i), (ii) hold.
Also, we  classify the Leonard
pairs $A, A^*$ on $V$ that satisfy (ii),
without assuming (i).
To this end,  it is convenient to bring in the following maps.
For $0 \leq i \leq d$ let $E^*_i : V \to V$ denote the projection onto the
$\th^*_i$-eigenspace of $A^*$.
An element $X \in \text{Span} \{ I,A,A^*, AA^*, A^*A\}$
is said to have zero diagonal whenever $E^*_i X E^*_i = 0$ for $0 \leq i \leq d$.
As we will see in Lemma \ref{lem:AAsmAsA}, the element $AA^*-A^*A$ has zero diagonal.
Let ${\mathcal Z}(A,A^*)$ denote the set
of elements in $\text{Span} \{I, A^*, A, AA^*\}$
that have zero diagonal. 
Note that  ${\mathcal Z}(A,A^*)$ is a subspace of  $\text{Span} \{I, A^*, A, AA^*\}$.
We call   ${\mathcal Z}(A,A^*)$ the zero diagonal space of $A,A^*$.
We show that the dimension of  ${\mathcal Z}(A,A^*)$ is at most $2$.
We show that ${\mathcal Z}(A,A^*)\not=0$ if and only 
if condition (ii) holds above if and only if
\[
 (a_i - a_0)(\th^*_i - \th^*_d)(a_j - a_d)(\th^*_j - \th^*_0)
= (a_i - a_d)(\th^*_i - \th^*_0)(a_j - a_0)(\th^*_j - \th^*_d)
\]
for $0 \leq i,j \leq d$.

As we describe ${\mathcal Z}(A,A^*)$  in
depth, we break the problem into 13 cases called types:
$q$-Racah,
$q$-Hahn,
dual $q$-Hahn,
quantum $q$-Krawtchouk,
$q$-Krawtchouk,
affine $q$-Krawtchouk,
dual $q$-Krawtchouk,
Racah,
Hahn,
dual Hahn,
Krawtchouk,
Bannai/Ito,
Orphan.
For each type $A,A^*$ is described by some parameters such as $s,s^*, r_1, r_2$.
For each type we give a necessary and sufficient condition on these parameters
for ${\mathcal Z}(A,A^*) \not=0$. 
For  each type
we give an explicit basis for ${\mathcal Z}(A,A^*)$.
The main results of this paper are Theorems \ref{thm:main1}--\ref{thm:main} and Theorem 
\ref{thm:characterize}.

This paper is organized as follows.
Section \ref{sec:prili} contains some preliminaries.
In Section \ref{sec:LP} we recall some basic results about a Leonard pair $A,A^*$.
In Sections \ref{sec:ZD}, \ref{sec:M} we introduce the zero diagonal space ${\mathcal Z}(A,A^*)$
and discuss its basic properties.
In Section \ref{sec:soleqZ} we describe  ${\mathcal Z}(A,A^*)$ using the scalars
$a^-_i = (a_i - a_0)(\th^*_i - \th^*_d)$ and
 $a^+_i = (a_i - a_d)(\th^*_i - \th^*_0)$ for $0 \leq i \leq d$.
In Section \ref{sec:types} we describe the $13$ types.
In Sections \ref{sec:neweq}, \ref{sec:main}, 
for each type we obtain a necessary and sufficient condition for ${\mathcal Z}(A,A^*)\neq 0$.
In Section \ref{sec:dimZ2} we describe the Leonard pairs $A,A^*$ such that ${\mathcal Z}(A,A^*)$ has dimension $2$.
In Sections \ref{sec:rel}, \ref{sec:rev}
we describe the Leonard pairs $A,A^*$ such that ${\mathcal Z}(A,A^*)$ has dimension $1$.
In Section \ref{sec:SD} we describe the self-dual condition.
In Section \ref{sec:spinLP} we show that a Leonard pair $A,A^*$ has spin if
and only if $A,A^*$ is self-dual and ${\mathcal Z}(A,A^*) \neq 0$.

\section{Preliminaries}
\label{sec:prili}
\ifDRAFT {\rm sec:preli}. \fi

Throughout the paper, 
the following assumptions and notational conventions are in force.
Let $\F$ denote an algebraically closed  field.
By a {\em scalar} we mean an element of $\F$.
Let  $V$ denote a vector space over $\F$ with finite positive dimension.
Let $\text{\rm End}(V)$ denote the $\F$-algebra consisting of the  $\F$-linear maps $V \to V$.
Let $I \in \text{\rm End}(V)$ denote the identity map.
An element $A \in \text{End}(V)$ is said to be {\em diagonalizable} whenever $V$ is spanned by
the eigenspaces of $A$.
An element $A \in \text{End}(V)$ is said to be {\em multiplicity-free} 
whenever $A$ is diagonalizable and each eigenspace of $A$  has dimension one.
Assume that $A$ is multiplicity-free,  and
let $\{V_i\}_{i=0}^d$ denote an ordering of the eigenspaces of $A$.
The sum $V = \sum_{i=0}^d V_i$ is direct.
For $0 \leq i \leq d$ let $\th_i$ denote the eigenvalue of $A$ for $V_i$.
The scalars  $\{\th_i\}_{i=0}^d$ are mutually distinct.
For $0 \leq i \leq d$ define $E_i \in \text{End}(V)$ such that
$(E_i - I) V_i = 0$ and 
$E_i V_j = 0$ if $j \neq i$ $(0 \leq j \leq d)$.
Then
(i) $ A E_i = E_i A = \th_i E_i$ $(0 \leq i \leq d)$;
(ii) $E_i E_j = \delta_{i,j}E_i$ $(0 \leq i,j \leq d)$;
(iii) $\sum_{i=0}^d E_i = I$;
(iv) $A =\sum_{i=0}^d \th_i E_i$.
We call  $E_i$ the {\em primitive idempotent of $A$ associated with $\th_i$} $(0 \leq i \leq d)$.
By linear algebra,
\begin{equation}
E_i = \prod_{\text{\scriptsize $\begin{matrix}  0 \leq j \leq d \\ j \neq i \end{matrix}$} }
           \frac{A-\th_j I} { \th_i - \th_j}
\qquad\qquad ( 0 \leq i \leq d).               \label{eq:Ei}
\end{equation}

Let $\F^{d+1}$ denote the vector space over $\F$
consisting of the row vectors with  $d+1$ coordinates and all entries in $\F$.

\section{Leonard pairs and Leonard systems}
\label{sec:LP}
\ifDRAFT {\rm sec:LP}. \fi

In this section, we recall the definition and basic properties
of a Leonard pair and a Leonard system.
For more information see \cite{T:Leonard}, \cite{T:survey}, \cite{T:notes}.

\begin{defi}    {\rm (See \cite[Definition 1.1]{T:Leonard}.) }
\label{def:LP}  \samepage
\ifDRAFT {\rm def:LP}. \fi
A {\em Leonard pair on $V$} is an ordered  pair $A, A^*$ of elements in $\text{\rm End}(V)$
that satisfy the following {\rm (i), (ii)}:
\begin{itemize}
\item[\rm (i)]
there exists a basis for $V$ with respect to which the matrix representing $A$ is irreducible
tridiagonal and the matrix representing $A^*$ is diagonal;
\item[\rm (ii)]
there exists a basis for $V$ with respect to which the matrix representing $A^*$ is irreducible
tridiagonal and the matrix representing $A$ is diagonal.
\end{itemize}
\end{defi}

\begin{lemma}   {\rm (See \cite[Lemma 1.3]{T:Leonard}.) }
\label{lem:mfree}    \samepage
Let $A,A^*$ denote a Leonard pair on $V$.
Then each of  $A, A^*$  is multiplicity-free.
\end{lemma}

\begin{lemma}   {\rm (See \cite[Corollary 5.5]{T:survey}.) }
\label{lem:generate}   \samepage
\ifDRAFT {\rm lem:generate}. \fi
Let $A,A^*$ denote a Leonard pair on $V$.
Then the elements $A,A^*$ together generate $\text{\rm End}(V)$.
\end{lemma}

We recall the notion of isomorphism for Leonard pairs.
Let $A,A^*$ denote a Leonard pair on $V$. Let $V'$ denote a vector space
over $\F$ with finite  positive  dimension, and let 
$A', A^{*\prime}$ denote a Leonard pair on $V'$.
By an {\em isomorphism of Leonard pairs from $A,A^*$ to $A', A^{*\prime}$}
we mean an $\F$-algebra isomorphism $\sigma  : \text{End}(V) \to \text{End}(V')$ such that
$A^\sigma = A'$ and $(A^*)^\sigma = A^{* \prime}$.
The Leonard pairs $A,A^*$ and $A', A^{* \prime}$ are said to be {\em isomorphic} whenever there exists an
isomorphism of Leonard pairs from $A,A^*$ to $A', A^{* \prime}$.

Next we recall the notion of a Leonard system.
We will use the following notation.
Let $A,A^*$ denote a Leonard pair on $V$.
Let $\{\th_i\}_{i=0}^d$\
 denote an ordering of the eigenvalues of $A$.
For $0 \leq i \leq d$ let $v_i$  denote an eigenvector for $A$
corresponding to $\th_i$.
We say that the ordering $\{\th_i\}_{i=0}^d$  is {\em standard}
whenever $\{v_i\}_{i=0}^d$
satisfies Definition \ref{def:LP}(ii).
For $0 \leq i \leq d$ let $E_i$ denote the primitive idempotent of $A$
associated with $\th_i$.
The ordering $\{E_i\}_{i=0}^d$  is said to be {\em standard}
whenever the ordering $\{\th_i\}_{i=0}^d$ is standard.
If an ordering $\{E_i\}_{i=0}^d$ is standard then the ordering $\{E_{d-i}\}_{i=0}^d$
is standard and no further ordering is standard.
A standard ordering of the primitive idempotents of $A^*$
is similarly defined.

\begin{defi}   {\rm (See \cite[Definition 3.1]{NT:sd}.) }
\label{def:LS}    \samepage
\ifDRAFT {\rm def:LS}. \fi
By a {\em Leonard system on $V$} we mean a sequence
\[
 \Phi = ( A; \{E_i\}_{i=0}^d; A^*; \{E^*_i\}_{i=0}^d)
\]
that satisfies the following {\rm (i)--(iii)}:
\begin{itemize}
\item[\rm (i)]
$A,A^*$ is a Leonard pair on $V$;
\item[\rm (ii)]
$\{E_i\}_{i=0}^d$ is a standard ordering of the primitive idempotents of $A$;
\item[\rm (iii)]
$\{E^*_i\}_{i=0}^d$ is a standard ordering of the primitive idempotents of $A^*$.
\end{itemize}
We say that the Leonard pair $A,A^*$ and the Leonard system $\Phi $ are {\em associated}.
\end{defi}

For the rest of this section, let
\[
 \Phi = ( A; \{E_i\}_{i=0}^d; A^*; \{E^*_i\}_{i=0}^d)
\]
denote a Leonard system on $V$.

\begin{lemma}    {\rm (See \cite[Lemma 3.3]{TV}.) }
\label{lem:trid0}   \samepage
\ifDRAFT {\rm lem:trid0}. \fi
We have
\begin{align*}
E_i A^* E_j &=
 \begin{cases}
   0 & \text{ if $|i-j|>1$},
\\ 
  \neq 0 & \text{ if $|i-j|= 1$}
\end{cases}
\qquad\qquad
( 0 \leq i,j \leq d);
\\
E^*_i A E^*_j &=
 \begin{cases}
   0 & \text{ if $|i-j|>1$},
  \\ 
  \neq 0 & \text{ if $|i-j|= 1$}
\end{cases}
\qquad\qquad
( 0 \leq i,j \leq d).
\end{align*}
\end{lemma}

As we will see, our main results are meaningful only for $d \geq 3$.
For the rest of this paper we assume that
$d \geq 3$.

Next we recall the notion of isomorphism for Leonard systems.
Let $V'$ denote a vector space over $\F$ with  dimension $d+1$.
For  an algebra isomorphism  $\sigma : \text{\rm End}(V) \to \text{\rm End}(V')$,
define
\[
\Phi^\sigma = (A^\sigma; \{E^\sigma_i\}_{i=0}^d; A^{*\sigma};
                          \{E^{*\sigma}_i\}_{i=0}^d).
\]
Then $\Phi^\sigma$ is a Leonard system on $V'$.
For a Leonard system  $\Phi'$ on $V'$,
by an {\em isomorphism of Leonard systems from $\Phi$ to $\Phi'$}
we mean an algebra isomorphism 
 $\sigma : \text{\rm End}(V) \to \text{\rm End}(V')$
such that $\Phi^\sigma = \Phi'$.
The Leonard systems $\Phi$ and $\Phi'$ are said to be {\em isomorphic}
whenever there exists an isomorphism of Leonard systems from $\Phi$ to $\Phi'$.

For scalars $\xi, \zeta, \xi^*, \zeta^* \in \F$ with $\xi \xi^* \neq 0$,
the sequence
\begin{equation}
    ( \xi A + \zeta I;  \{E_i\}_{i=0}^d; \xi^*  A^* + \zeta^* I; \{E^*_i\}_{i=0}^d)     \label{eq:affine}
\end{equation}
is a Leonard system on $V$, called an {\em affine transformation of $\Phi$}.
Also,
\begin{align*}
\Phi^* = ( A^*; \{E^*_i\}_{i=0}^d; A; \{E_i\}_{i=0}^d),
\\
\Phi^\downarrow = ( A; \{E_i\}_{i=0}^d; A^*; \{E^*_{d-i}\}_{i=0}^d),
\\
\Phi^\Downarrow = ( A; \{E_{d-i}\}_{i=0}^d; A^*; \{E^*_i\}_{i=0}^d)
\end{align*}
are Leonard systems on $V$.
By construction, the Leonard systems associated with $A,A^*$ are $\Phi$, $\Phi^\downarrow$,
$\Phi^\Downarrow$, $\Phi^{\downarrow \Downarrow}$.

Next we recall the parameter array of a Leonard system.

\begin{defi}     {\rm (See \cite[Definition 1.8]{T:Leonard}.) }
\label{def:eigenseq}    \samepage
\ifDRAFT {\rm def:eigenseq}. \fi
For $0 \leq i \leq d$ let $\th_i$ (resp.\ $\th^*_i$)
denote the eigenvalue of $A$ (resp.\ $A^*$)
associated with $E_i$ (resp.\ $E^*_i$).
We call $\{\th_i\}_{i=0}^d$ (resp.\ $\{\th^*_i\}_{i=0}^d$)
the {\em eigenvalue sequence} (resp.\ {\em dual eigenvalue sequence})
of $\Phi$.
\end{defi}

We have some comments about the eigenvalues and dual eigenvalues.
We have
\[
A E_i = \th_i E_i = E_i A, 
\qquad\qquad
A^* E^*_i =  \th^*_i E^*_i= E^*_i A^*
\qquad\qquad
(0 \leq i \leq d).
\]
Let $x$ denote an indeterminate.
For $0 \leq i \leq d$ define a polynomial
\[
\tau_i (x) = (x-\th_0) (x-\th_1) \cdots (x-\th_{i-1}).
\]
Note that $\tau_i(x)$ is monic with degree $i$.

\begin{defi}   {\rm (See \cite[Section 21]{T:survey}.) }
\label{def:splitbasis}   \samepage
\ifDRAFT {\rm def:splitbasis}. \fi
Pick $0 \neq u \in E^*_0 V$.
For $0 \leq i \leq d$ define the vector
\[
  u_i = \tau_i(A)  u.
\]
The  vectors $\{u_i\}_{i=0}^d$ form a basis for $V$.
This basis is said to be {\em $\Phi$-split}.
\end{defi}

\begin{lemma}    {\rm (See \cite[Section 21]{T:survey}.) }
\label{lem:defvphi}    \samepage
\ifDRAFT {\rm lem:defvphi}. \fi
With respect to a $\Phi$-split basis of $V$,
the matrices representing $A$ and $A^*$ have the form
\[
A :
\begin{pmatrix}
\th_0  & & & & & {\bf 0} \\
 1 & \th_1 \\
    & 1 & \th_2  \\
  &   &  \cdot & \cdot \\
  &   &   &  \cdot & \cdot  \\
{\bf 0} & & & &  1 & \th_d
\end{pmatrix},
\qquad\qquad
A ^* :
\begin{pmatrix}
\th^*_0  & \vphi_1 & & & & {\bf 0} \\
 & \th^*_1  & \vphi_2 \\
    &  & \th^*_2 & \cdot \\
  &   & &  \cdot & \cdot \\
  &   &   & &  \cdot &  \vphi_d  \\
{\bf 0} & & & &   & \th^*_d
\end{pmatrix},
\]
where $0 \neq \vphi_i \in \F$ for $1 \leq i \leq d$.
\end{lemma}

\begin{defi}    {\rm (See \cite[Section 8]{T:TDD}.)}
\label{def:splitseq}  \samepage
\ifDRAFT {\rm def:splitseq}. \fi
Referring to Lemma \ref{lem:defvphi}, 
the sequence $\{\vphi_i\}_{i=1}^d$ is called the {\em first split sequence of $\Phi$}.
By the {\em second split sequence of $\Phi$} we mean the
first split sequence of $\Phi^\Downarrow$.
\end{defi}

\begin{defi}   {\rm (See \cite[Definition 11.1]{T:TDD}.) }
\label{def:parray}   \samepage
\ifDRAFT {\rm def:parray}. \fi
By the {\em parameter array of $\Phi$} we mean
the sequence
\[
  (\{\th_i\}_{i=0}^d;
   \{\th^*_i\}_{i=0}^d;
   \{\vphi_i\}_{i=1}^d;
  \{\phi_i\}_{i=1}^d),
\]
where
$\{\vphi_i\}_{i=1}^d$  (resp.\  $\{\phi_i\}_{i=1}^d$)  is the first split sequence 
(resp. second split sequence) of $\Phi$.
\end{defi}

\begin{lemma}   {\rm (See \cite[Theorem 22.2]{T:survey}.) }
\label{lam:LPclassify}    \samepage
\ifDRAFT {\rm lem:LPclassify}. \fi
The  Leonard system $\Phi$ is uniquely determined up to isomorphism 
by its parameter array.
\end{lemma}

\begin{note}   \samepage
For a detailed description of the parameter arrays,
see \cite[Theorem 10.1]{T:notes} and \cite[Appendix]{T:notes}.
\end{note}

\begin{lemma}   {\rm (See \cite[Lemma 5.1]{NT:affine}.) }
\label{lem:affine}  \samepage
\ifDRAFT {\rm lem:affine}. \fi
Consider the affine transformation of $\Phi$ shown in \eqref{eq:affine}.
This affine transformation  has parameter array
\[
 (\{ \xi \th_i + \zeta\}_{i=0}^d;  \{\xi^* \th^*_i + \zeta^*\}_{i=0}^d;
   \{\xi \xi^* \vphi_i\}_{i=1}^d;
   \{\xi \xi^* \phi_i\}_{i=1}^d).
\]
\end{lemma}

\begin{lemma}     {\rm (See \cite[Theorem 1.11]{T:Leonard}.) }
\label{lem:parray}   \samepage
\ifDRAFT {\rm lem:parray}. \fi
The following hold:
\begin{itemize}
\item[\rm (i)]
the parameter array of $\Phi^\downarrow$ is
\[
(\{\th_i\}_{i=0}^d;  \{\th^*_{d-i}\}_{i=0}^d;
 \{\phi_{d-i+1}\}_{i=1}^d;  \{\vphi_{d-i+1}\}_{i=1}^d);
\]
\item[\rm (ii)]
the parameter array of $\Phi^\Downarrow$ is
\[
(\{\th_{d-i}\}_{i=0}^d;  \{\th^*_{i}\}_{i=0}^d;
 \{\phi_i\}_{i=1}^d;  \{\vphi_i\}_{i=1}^d);
\]
\item[\rm (iii)]
the parameter array of $\Phi^*$ is
\[
(\{\th^*_i\}_{i=0}^d;  \{\th_i\}_{i=0}^d;
 \{\vphi_i\}_{i=1}^d;  \{\phi_{d-i+1}\}_{i=1}^d).
\]
\end{itemize}
\end{lemma}

\begin{defi}   {\rm (See \cite[Definition 7.1]{T:survey}.)}
 \label{def:ai}   \samepage
\ifDRAFT {\rm def:ai}. \fi
Define
\[
a_i = \text{tr} (E^*_i A)
 \qquad\qquad ( 0 \leq i \leq d),
\]
where tr means trace.
\end{defi}

\begin{lemma}   {\rm (See \cite[Lemma 7.5]{T:survey}.) }
\label{lem:ai}   \samepage
\ifDRAFT {\rm lem:ai}. \fi
We have
\[
E^*_i A E^*_i = a_i E^*_i
\qquad\qquad (0 \leq i \leq d).
\]
\end{lemma}

\begin{lemma}   {\rm (See \cite[Lemma 23.6]{T:survey}.) }
\label{lem:aiformula}   \samepage
\ifDRAFT {\rm lem:aiformula}. \fi
We have
\begin{align*}
a_0 &= \th_0 + \frac{\vphi_1}{\th^*_0 - \th^*_1},
\\
a_i &= \th_i + \frac{\vphi_i} { \th^*_i - \th^*_{i-1}} + \frac{\vphi_{i+1} } { \th^*_i - \th^*_{i+1} }
        \qquad \qquad (1 \leq i \leq d-1),
\\
a_d &= \th_d + \frac{\vphi_d} { \th^*_d - \th^*_{d-1}}.
\end{align*}
\end{lemma}

\begin{defi}   {\rm (See \cite[Definition 10.3]{T:survey}.) }
\label{def:Phistandard}    \samepage
\ifDRAFT {\rm def:Phistandard}. \fi
Pick $0 \neq u \in E_0V$.
For $0 \leq i \leq d$ define the vector
\[
  v^*_i  = E^*_i u.
\]
The vectors $\{v^*_i\}_{i=0}^d$ form a basis for $V$.
This basis is said to be {\em $\Phi$-standard}.
\end{defi}

\begin{lemma}   {\rm (See \cite[Definition 11.1]{T:survey}.) }
 \label{lem:defbici}   \samepage
\ifDRAFT {\rm lem:defbici}. \fi
With respect to a $\Phi$-standard basis for $V$,
the matrices representing $A$ and  $A^*$ have the form
\[
A :
\begin{pmatrix}
a_0 & b_0  & & & & {\bf 0}  \\
c_1 & a_1 & b_1 \\
 & c_2 & \cdot & \cdot \\
 &  & \cdot & \cdot & \cdot \\
  & & & \cdot & \cdot & b_{d-1} \\
{\bf 0} & & & & c_d & a_d
\end{pmatrix},
\qquad\qquad
A^* :
\begin{pmatrix}
\th^*_0 &   & & & & {\bf 0}  \\
 & \th^*_1  \\
 &  & \th^*_2  \\
 &  &  & \cdot   \\
  & & &  & \cdot \\
{\bf 0} & & & &  & \th^*_d
\end{pmatrix},
\]
where $\{a_i\}_{i=0}^d$ are from Definition \ref{def:ai} 
and the scalars $\{c_i\}_{i=1}^d$,
$\{b_i\}_{i=0}^{d-1}$ are nonzero.
\end{lemma}

The scalars
$\{c_i\}_{i=1}^d$, 
 $\{a_i\}_{i=0}^d$, $\{b_i\}_{i=0}^{d-1}$
in Lemma \ref{lem:defbici} are called
the {\em intersection numbers of $\Phi$}. 
See \cite{T:Leonard}, \cite{T:survey}, \cite{T:notes}
for basic facts about the intersection numbers.


\section{The zero diagonal space for a Leonard pair }
\label{sec:ZD}
\ifDRAFT {\rm sec:ZD}. \fi

In this section, to each Leonard pair $A, A^*$
on $V$ we associate a subspace of $\text{End}(V)$
called the zero diagonal space for $A, A^*$.

\begin{defi}   \label{def:X}   \samepage
\ifDRAFT {\rm def:X}. \fi
For a Leonard system $\Phi = ( A; \{E_i\}_{i=0}^d; A^*; \{E^*_i\}_{i=0}^d)$ on $V$,
let ${\mathcal X}(\Phi)$ denote the subspace of $\text{End}(V)$
consisting of the $X \in \text{End}(V)$ such that
\[
E_i X E_j = 0, 
\qquad 
E^*_i X E^*_j = 0 
\qquad 
\text{ if $|i-j| > 1$} 
\qquad
 (0 \leq i,j \leq d).
\]
\end{defi}

\begin{lemma}   \label{lem:Xaffine}   \samepage
\ifDRAFT {\rm lem:Xaffine}. \fi
Let
$\Phi = ( A; \{E_i\}_{i=0}^d; A^*; \{E^*_i\}_{i=0}^d)$
denote a Leonard system on $V$,
and let $\Phi'$ denote an affine transformation of $\Phi$ from \eqref{eq:affine}.
Then ${\mathcal X} ( \Phi') = {\mathcal X}(\Phi)$.
\end{lemma}

\begin{proof}
The  $\Phi$ and $\Phi'$ have the same $\{E_i\}_{i=0}^d$ and
 $\{E^*_i\}_{i=0}^d$.
\end{proof}

\begin{lemma}   \label{lem:XAAs}   \samepage
\ifDRAFT {\em lem:XAAs}. \fi
Let
$\Phi = ( A; \{E_i\}_{i=0}^d; A^*; \{E^*_i\}_{i=0}^d)$
denote a Leonard system on $V$.
Then 
${\mathcal X}(\Phi ) = {\mathcal X}(\Phi^\downarrow) = {\mathcal X}(\Phi^\Downarrow)$.
\end{lemma}

\begin{proof}
By Definition \ref{def:X}.
\end{proof}

In view of Lemma \ref{lem:XAAs}, we make a definition.

\begin{defi}   \label{def:XLP}   \samepage
\ifDRAFT {\rm def:XLP}. \fi
For a Leonard pair $A,A^*$ on $V$ 
we define ${\mathcal X}(A,A^*) = {\mathcal X}(\Phi)$,
where $\Phi$ is a Leonard system associated with $A,A^*$.
\end{defi}

\begin{lemma}   {\rm (See \cite[Theorem 3.2]{NT:trid}.) }
\label{lem:trid}   \samepage
\ifDRAFT {\rm lem:trid}. \fi
Let
$A, A^*$ denote a Leonard pair on $V$.
Then the following elements  form a basis for ${\mathcal X}(A, A^*)$:
\[
  I, \, A^*,\,  A,\,  A A^*, \, A^* A.   
\]
\end{lemma}

\begin{defi}   \label{def:zerodiagonal}   \samepage
\ifDRAFT {\rm def:zeriduagibal}. \fi
Let
$\Phi = ( A; \{E_i\}_{i=0}^d; A^*; \{E^*_i\}_{i=0}^d)$
denote a Leonard system on $V$.
An element $X \in {\mathcal X}(\Phi)$ is said to have {\em zero diagonal}
whenever $E^*_i X E^*_i = 0$ for $0 \leq i \leq d$.
\end{defi}

\begin{lemma}   \label{lem:AAsmAsA}   \samepage
\ifDRAFT {\rm lem:AAsmAsA}. \fi
Let
$\Phi = ( A; \{E_i\}_{i=0}^d; A^*; \{E^*_i\}_{i=0}^d)$
denote a Leonard system on $V$.
Then the element $A A^* - A^* A$ has zero diagonal.
\end{lemma}

\begin{proof}
Let $\{\th^*_i\}_{i=0}^d$ denote the dual eigenvalue sequence of $\Phi$.
For $0 \leq i \leq d$ we have
$E^*_i A^* = \th^*_i E^*_i$ and
$A^* E^*_i = \th^*_i E^*_i$.
Therefore
\[
E^*_i (A A^* - A^* A) E^*_i
= E^*_i A A^* E^*_i - E^*_i A^* A E^*_i
= \th^*_i E^*_i A E^*_i -  \th^*_i E^*_i A E^*_i = 0.
\]
\end{proof}

\begin{lemma}    \label{lem:newbasis}  \samepage
\ifDRAFT {\rm lemk:newbasis}. \fi
Let $A,A^*$ denote a Leonard pair on $V$.
Then the  following elements form a basis for ${\mathcal X}(A,A^*)$.
\[
I,  \;
A^*,  \;
A,  \;
A A^*,  \;
A A^* - A^* A.
\]
\end{lemma}

\begin{proof}
By  Lemma \ref{lem:trid}.
\end{proof}

\begin{defi}   \label{def:Z}   \samepage
\ifDRAFT {\rm def:Z}. \fi
Let
$\Phi = ( A; \{E_i\}_{i=0}^d; A^*; \{E^*_i\}_{i=0}^d)$
denote a Leonard system on $V$.
Let ${\mathcal Z}(\Phi)$ denote the set of elements in
$\text{\rm Span}\{ I, A^*, A, A A^*\}$
that have zero diagonal.
Note that ${\mathcal Z}(\Phi)$ is a subspace of
$\text{\rm Span}\{ I, A^*, A, A A^*\}$.
We call ${\mathcal Z}(\Phi)$ the
{\em zero diagonal space} for $\Phi$.
\end{defi}

\begin{lemma}   \label{lem:decomp}    \samepage
\ifDRAFT {\rm lem:decomp}. \fi
Let
$\Phi = ( A; \{E_i\}_{i=0}^d; A^*; \{E^*_i\}_{i=0}^d)$
denote a Leonard system on $V$.
Then for $X \in {\mathcal X}(\Phi)$ the following are equivalent:
\begin{itemize}
\item[\rm (i)]
$X$ has zero diagonal;
\item[\rm (ii)]
$X \in \text{\rm Span} \{A A^* - A^* A\} + {\mathcal Z}(\Phi)$.
\end{itemize}
\end{lemma}

\begin{proof}
By Lemmas \ref{lem:AAsmAsA}, \ref{lem:newbasis}
and Definition \ref{def:Z}.
\end{proof}

\begin{lemma}   \label{lem:Zaffine}   \samepage
\ifDRAFT {\rm lem:Zaffine}. \fi
Let
$\Phi = ( A; \{E_i\}_{i=0}^d; A^*; \{E^*_i\}_{i=0}^d)$
denote a Leonard system on $V$.
Let $\Phi'$ denote an affine transformation of $\Phi$ from \eqref{eq:affine}.
Then ${\mathcal Z} ( \Phi') = {\mathcal Z}(\Phi)$.
\end{lemma}

\begin{proof}
The  $\Phi$ and $\Phi'$ have the same 
 $\{E^*_i\}_{i=0}^d$.
\end{proof}

\begin{lemma}   \label{lem:ZAAs}   \samepage
\ifDRAFT {\em lem:ZAAs}. \fi
Let
$\Phi = ( A; \{E_i\}_{i=0}^d; A^*; \{E^*_i\}_{i=0}^d)$
denote a Leonard system on $V$.
Then  ${\mathcal Z}(\Phi ) = {\mathcal Z}(\Phi^\downarrow) = {\mathcal Z}(\Phi^\Downarrow)$.
\end{lemma}

\begin{proof}
By Definition \ref{def:Z}.
\end{proof}

In view of Lemma \ref{lem:ZAAs}, we make a definition.

\begin{defi}   \label{def:ZLP}   \samepage
\ifDRAFT {\rm def:ZLP}. \fi
For a Leonard pair $A,A^*$ on $V$ 
we define ${\mathcal Z}(A,A^*) = {\mathcal Z}(\Phi)$,
where $\Phi$ is a Leonard system associated with $A,A^*$.
We call  ${\mathcal Z}(A,A^*)$ the {\em zero diagonal space for $A,A^*$}.
\end{defi}

We mention a result for later use.

\begin{lemma}   \label{lem:linindep}    \samepage
\ifDRAFT {\rm lem:linindep}. \fi
Let
$\Phi = ( A; \{E_i\}_{i=0}^d; A^*; \{E^*_i\}_{i=0}^d)$
denote a Leonard system on $V$.
Then the following elements are linearly independent:
\[
(A-a_0 I)(A^* - \th^*_d I),
\qquad\qquad
(A-a_d I)(A^* - \th^*_0 I).
\]
\end{lemma}

\begin{proof}
By Lemma \ref{lem:trid}
the elements $I, A^*, A,A A^*$ are linearly independent.
The result follows from this and $\th^*_0 \neq \th^*_d$.
\end{proof}

\section{The matrix $M$}
\label{sec:M}
\ifDRAFT {\rm secM}. \fi

For the rest of this paper, we fix a Leonard system
\begin{equation}
\Phi = ( A; \{E_i\}_{i=0}^d; A^*; \{E^*_i\}_{i=0}^d)     \label{eq:Phi}
\end{equation}
on $V$ with parameter array
\[
(\{\th_i\}_{i=0}^d; \{\th^*_i\}_{i=0}^d; \{\vphi_i\}_{i=1}^d; \{\phi_i\}_{i=1}^d).
\]
Let the scalars $\{a_i\}_{i=0}^d$ be from Definition \ref{def:ai}.
In this section we introduce a matrix $M$,
and explain how the rank of $M$ is related to the dimension of ${\mathcal Z}(\Phi)$.

\begin{lemma}   \label{lem:cond}   \samepage
\ifDRAFT {\rm lem:cond}. \fi
For scalars $f_0,f_1,f_2,f_3$ the following are equivalent:
\begin{itemize}
\item[\rm (i)]
$f_0 I + f_1 A^* + f_2 A + f_3 A A^* \in {\mathcal Z}(\Phi)$;
\item[\rm (ii)]
for $0 \leq i  \leq d$,
\[
  f_0 + f_1 \th^*_i  + f_2 a_i + f_3 a_i \th^*_i = 0. 
\]
\end{itemize}
\end{lemma}

\begin{proof}
For $0 \leq i \leq d$ we have
\[
E^*_i (f_0 I + f_1 A^* + f_2 A + f_3 A A^*) E^*_i
 = (f_0 + f_1 \th^*_i + f_2 a_i + f_3 a_i \th^*_i) E^*_i.
\]
The result follows.
\end{proof}

\begin{lemma}   \label{lem:cond2}   \samepage
\ifDRAFT {\rm lem:cond2}. \fi
The following are equivalent:
\begin{itemize}
\item[\rm (i)]
${\mathcal Z}(\Phi) \neq 0$;
\item[\rm (ii)]
there exist scalars $f_0,f_1,f_2,f_3$ (not all zero) such that
\begin{equation}
  f_0 + f_1 \th^*_i  + f_2 a_i + f_3 a_i \th^*_i = 0
\qquad\qquad (0 \leq i \leq d).
  \label{eq:Z}
\end{equation}
\end{itemize}
\end{lemma}

\begin{proof}
By Lemma \ref{lem:cond}.
\end{proof}

\begin{defi}   \label{def:M}   \samepage
\ifDRAFT {\rm def:M}. \fi
Define a matrix  $M = M(\Phi)$  of size $4 \times (d+1)$ by
\[
M =
\begin{pmatrix}
1 & 1 & \cdots  & 1
\\
\th^*_0 & \th^*_1 & \dots & \th^*_d
\\
a_0  & a_1 & \cdots & a_d
\\
a_0 \th^*_0 & a_1 \th^*_1 & \cdots & a_d \th^*_d
\end{pmatrix}.     
\]
\end{defi}

\begin{lemma}   \label{lem:rankM}    \samepage
\ifDRAFT {\rm lem:rankM}. \fi
We have
$2 \leq \text{\rm rank}(M) \leq 4$.
\end{lemma}

\begin{proof}
Clearly $\text{rank}(M) \leq 4$.
The top two rows of $M$ are linearly independent, since $\{\th^*_i\}_{i=0}^d$
are mutually distinct.
Therefore  $\text{rank}(M) \geq 2$.
\end{proof}

\begin{defi}    \label{def:psi}   \samepage
\ifDRAFT {\rm def:psi}. \fi
Define an $\F$-linear map 
\[
\psi : \text{\rm Span}\{I,A^*,A,A A^*\}  \to \F^{d+1}
\]
that sends
\[
 f_0 I + f_1 A^* + f_2 A + f_3 A A^* \mapsto  (f_0, f_1, f_2, f_3) M.
\]
\end{defi}

\begin{lemma}   \label{lem:Impsi}   \samepage
\ifDRAFT {\rm lem:Impsi}. \fi
The following hold:
\begin{itemize}
\item[\rm (i)]
$\text{\rm Ker}(\psi) = {\mathcal Z}(\Phi)$;
\item[\rm (ii)]
$\dim \text{\rm Im} (\psi) = \text{\rm rank}(M)$;
\item[\rm (iii)]
$\dim  {\mathcal Z}(\Phi) = 4 - \text{\rm rank}(M)$.
\end{itemize}
\end{lemma}

\begin{proof}
(i)
By Lemma \ref{lem:cond}.

(ii), (iii)
By linear algebra.
\end{proof}

\begin{lemma}   \label{lem:dimZpre}   \samepage
\ifDRAFT {\rm lem:dimZpre}. \fi
We have
\[
 0 \leq \dim {\mathcal Z}(\Phi) \leq 2.
\]
\end{lemma}

\begin{proof}
By Lemmas \ref{lem:rankM} and  \ref{lem:Impsi}(iii).
\end{proof}

\begin{lemma}    \label{lem:condd}   \samepage
\ifDRAFT {\rm lem:condd}. \fi
The following are equivalent:
\begin{itemize}
\item[\rm (i)]
$\dim  {\mathcal Z}(\Phi) = 2$;
\item[\rm (ii)]
$\text{\rm rank}(M) = 2$;
\item[\rm (iii)]
$a_0 = a_1 = \cdots = a_d$.
\end{itemize}
\end{lemma}

\begin{proof}
(i) $\Leftrightarrow$ (ii)
By Lemma \ref{lem:Impsi}(iii).

(ii) $\Rightarrow$ (iii)
The top two rows of $M$ are linearly independent.
So the 3rd row and the 4th row of $M$ are contained in the span of the
top two rows.
Thus there exist scalars $\alpha$, $\beta$, $\gamma$, $\delta$ such that
\begin{align*}
a_i &= \alpha + \beta\th^*_i   \qquad\qquad (0 \leq i \leq d),
\\
a_i \th^*_i &= \gamma + \delta \th^*_i   \qquad\qquad (0 \leq i \leq d).
\end{align*}
By these equations we get 
\[
 0 = \beta \th_i^{*2} + (\alpha-\delta) \th^*_i - \gamma
\qquad\qquad
(0 \leq i \leq d).
\]
We assume $d \geq 3$ and the scalars $\{\th^*_i\}_{i=0}^d$ are mutually distinct, so
$\beta = 0$, $\alpha= \delta$, $\gamma =0$.
It follows that $a_0 = a_1 = \cdots = a_d$.

(iii)  $\Rightarrow$ (ii)
Clear.
\end{proof}

\begin{lemma}   \label{lem:EB}   \samepage
\ifDRAFT {\rm lem:EB}. \fi
Assume that $\dim  {\mathcal Z}(\Phi) = 2$.
Then the scalars  $f_0,f_1,f_2,f_3$ satisfy \eqref{eq:Z} if and only if
\[
f_0 + f_2 a_0 = 0,
\qquad\qquad
f_1 + f_3 a_0 = 0.
\]
\end{lemma}

\begin{proof}
By Lemmas \ref{lem:cond} and \ref{lem:condd}.
\end{proof}

\begin{lemma}   \label{lem:EB2}   \samepage
\ifDRAFT {\rm lem:EB2}. \fi
Assume that $\dim  {\mathcal Z}(\Phi) = 2$.
Then the following elements form a basis for
${\mathcal Z}(\Phi)$:
\[
    A - a_0 I,
\qquad \qquad
A A^* - a_0 A^*.  
\]
\end{lemma}
 
\begin{proof}
By Lemmas \ref{lem:cond} and \ref{lem:EB}.
\end{proof}

\section{ The scalars $a^-_i$ and $a^+_i$}
\label{sec:soleqZ}
\ifDRAFT {\rm sec:soleqZ}. \fi

We continue to discuss the Leonard system $\Phi$ from \eqref{eq:Phi}.
Recall the zero diagonal space ${\mathcal Z}(\Phi)$ from Definition \ref{def:Z}.
In this section,
we obtain a necessary and sufficient condition for ${\mathcal Z}(\Phi) \neq 0$.

\begin{defi}   \label{def:amap}   \samepage
\ifDRAFT {\rm def:amap}. \fi
For $0 \leq i \leq d$ define
\begin{align*}
a^-_i &= (a_i - a_0)(\th^*_i - \th^*_d),
\qquad\qquad
a^+_i  = (a_i -a_d)(\th^*_i - \th^*_0).
\end{align*}
\end{defi}

\begin{lemma}   \label{lem:bemu}   \samepage
\ifDRAFT {\rm lem:bemu}. \fi
We have
$a^-_0 = a^-_d = a^+_0 = a^+_d = 0$.
\end{lemma}

\begin{proof}
By Definition \ref{def:amap}.
\end{proof}

\begin{defi}  \label{def:L}    \samepage
\ifDRAFT {\rm def:L}. \fi
Define a matrix $L = L(\Phi)$ of size $4 \times (d+1)$  by
\[
L =
\begin{pmatrix}
1 & 1 & \cdots & 1  
\\
\th^*_0 & \th^*_1 & \cdots & \th^*_d
\\
a^-_0 & a^-_1 & \cdots & a^-_d
\\
a^+_0 & a^+_1 & \cdots & a^+_d
\end{pmatrix}.
\]
\end{defi}

\begin{defi}  \label{def:T}    \samepage
\ifDRAFT {\rm def:T}. \fi
Define a matrix $T$  by
\[
T =
\begin{pmatrix}
1 & 0 & 0 & 0
\\
0 & 1 & 0 & 0
\\
a_0 \th^*_d & - a_0  & -\th^*_d& 1
\\
a_d \th^*_0  & - a_d & - \th^*_0& 1
\end{pmatrix}.
\]
\end{defi}

\begin{lemma}   \label{lem:LTM}    \samepage
\ifDRAFT {\rm lem:LTM}. \fi
We have
\[
  L = T M.
\]
\end{lemma}

\begin{proof}
Routine verification.
\end{proof}

\begin{lemma}    \label{lem:detT}   \samepage
\ifDRAFT {\rm lem:detT}. \fi
The matrix $T$ is invertible.
\end{lemma}

\begin{proof}
We have 
$ \det T = \th^*_0 - \th^*_d \neq 0$.
\end{proof}

\begin{corollary}   \label{cor:ML}    \samepage
\ifDRAFT {\rm cor:ML}. \fi
The matrices $M$ and $L$ have the same rank.
\end{corollary}

\begin{proof}
By Lemma \ref{lem:LTM} and linear algebra.
\end{proof}
 
\begin{prop}    \label{prop:cond}   \samepage
\ifDRAFT {\rm prop:cond}. \fi
The following are equivalent:
\begin{itemize}
\item[\rm (i)]
${\mathcal Z}(\Phi) \neq 0$;
\item[\rm (ii)]
$\text{\rm rank}(L) \leq 3$;
\item[\rm (iii)]
the vectors
$(a^-_0, a^-_1, \ldots, a^-_d)$
and
$(a^+_0, a^+_1, \ldots, a^+_d)$
are linearly dependent;
\item[\rm (iv)]
$a^-_i a^+_j = a^+_i a^-_j$ 
for $0 \leq i,j \leq d$;
\item[\rm (v)]
the vectors
$(a^-_1, a^-_ 2, \ldots, a^-_{d-1})$
and
$(a^+_1, a^+_1, \ldots, a^+_{d-1})$
are linearly dependent;
\item[\rm (vi)]
$a^-_i a^+_j = a^+_i a^-_j$ 
for $1 \leq i,j \leq d-1$.
\end{itemize}
\end{prop}

\begin{proof}
(i) $\Leftrightarrow$ (ii)
By Lemma \ref{lem:Impsi}(iii) and Corollary \ref{cor:ML}.

(ii) $\Rightarrow$ (iii)
The rows of $L$ are linearly dependent.
Therefore there exist scalars $f_0, f_1, f_2, f_3$ (not all zero)
such that
\[
f_0 + f_1 \th^*_i + f_2 a^-_i + f_3 a^+_i = 0
\qquad\qquad (0 \leq i \leq d).
\]
By this and since $a^-_0 = a^-_d = a^+_0 = a^+_d=0$,
\[
 f_0 + f_1 \th^*_0 = 0,
\qquad\qquad
 f_0 + f_1 \th^*_d = 0.
\]
We have $\th^*_0 \neq \th^*_d$,
so $f_0 = f_1 = 0$.
Observe that $f_2 \neq 0$ or $f_3 \neq 0$.
Moreover
\[
f_2 a^-_i + f_3 a^+_i = 0 \qquad\qquad (0 \leq i \leq d).
\]
This gives the linear dependency
\[
 f_2 (a^-_0, a^-_1, \ldots, a^-_d) + f_3 (a^+_0, a^+_1, \ldots, a^+_d) = 0.
\]

(iii) $\Rightarrow$ (ii)
By linear algebra.

(iii) $\Leftrightarrow$ (iv)
By linear algebra.

(iii) $\Leftrightarrow$ (v)
Since  $a^-_0 = a^-_d = a^+_0 = a^+_d=0$.

(iv) $\Leftrightarrow$ (vi)
Since $a^-_0 =a^-_d = a^+_0 = a^+_d = 0$.
\end{proof}

\section{The type of  a Leonard system}
\label{sec:types}
\ifDRAFT {\rm sec:types}. \fi

We continue to discuss the Leonard system $\Phi$ from \eqref{eq:Phi}.
By \cite[Theorem  5.16]{T:parray}, 
$\Phi$ has one of the following types:

\begin{tabular}{lllll}
$q$-Racah,
&
$q$-Hahn,
&
dual $q$-Hahn,
&
quantum $q$-Krawtchouk,
\\
$q$-Krawtchouk,
&
affine $q$-Krawtchouk,
&
dual $q$-Krawtchouk,
\\
Racah,
&
Hahn,
&
dual Hahn,
&
Krawtchouk,
\\
Bannai/Ito,
&
Orphan.
\end{tabular}

\vspace{1ex}

\noindent
In this section, we describe each type in detail.

\begin{defi}   
{\rm (See \cite[Example 5.3]{T:parray}.) }
\label{def:qRacah}   \samepage
\ifDRAFT {\rm def:qRacah}. \fi
The Leonard system
$\Phi$ is said to have  {\em $q$-Racah type}
whenever 
there exist scalars $q, h, h^*, s, s^*, r_1, r_2$ such that 
\begin{itemize}
\item[\rm (i)]
each of $q, h, h^*,s, s^*, r_1, r_2$ is nonzero;
\item[\rm (ii)] 
$r_1 r_2 = s s^* q^{d+1}$;
\item[\rm (iii)]
none of $q^i, r_1 q^i, r_2 q^i, s^* q^i/r_1, s^* q^i /r_2$ is equal to $1$ for $1 \leq i \leq d$;
\item[(iv)]
$s q^i \neq 1$, $s^* q^i \neq 1$ for $2 \leq i \leq 2d$;
\item[\rm (v)]
for $0 \leq i \leq d$,
\begin{align*}
\th_i &= \th_0 + h(1-q^i)(1- s q^{i+1}) q^{-i},
\\
\th^*_i &= \th^*_0 + h^*(1-q^i)(1- s^* q^{i+1}) q^{-i};
\end{align*}
\item[\rm (vi)]
for $1 \leq i  \leq d$,
\begin{align*}
\vphi_i &= h h^* q^{1-2i} (1-q^i)(1-q^{i-d-1}) (1- r_1 q^i)(1-r_2 q^i),
\\
\phi_i &= h h^* q^{1-2i} (1-q^i)(1-q^{i-d-1}) (r_1- s^* q^i)(r_2 - s^* q^i)/s^*.
\end{align*}
\end{itemize}
\end{defi}

\begin{defi}  {\rm (See \cite[Example 5.4]{T:parray}.) }
\label{def:qHahn}  \samepage
\ifDRAFT {\rm def:qHahn}. \fi
The Leonard system
$\Phi$ is said to have  {\em $q$-Hahn type}
whenever 
there exist scalars $q, h, h^*, s^*, r$ such that 
\begin{itemize}
\item[\rm (i)]
each of $q, h, h^*,  s^*, r$ is nonzero;
\item[\rm (ii)]
none of $q^i, r q^i, s^* q^i/r$ is equal to $1$ for $1 \leq i \leq d$;
\item[\rm (iii)]
$ s^* q^i \neq 1$ for $2 \leq i \leq 2d$;
\item[\rm (iv)]
for $0 \leq i \leq d$,
\begin{align*}
\theta_i &= \theta_0+h(1-q^i)q^{-i}, 
\\
\theta^*_i &=  \theta^*_0+h^*(1-q^i)(1-s^*q^{i+1})q^{-i};
\end{align*}
\item[\rm (v)]
for $1 \leq i \leq d$,
\begin{align*}
\varphi_i &= h h^*q^{1-2i}(1-q^i)(1-q^{i-d-1})(1-rq^i),
\\
\phi_i &=- h h^*q^{1-i}(1-q^i)(1-q^{i-d-1})(r-s^*q^i).
\end{align*}
\end{itemize}
\end{defi}

\begin{defi}  {\rm  (See \cite[Example 5.5]{T:parray}.) }
\label{def:dualqHahn}     \samepage
\ifDRAFT {\rm def:dualqHahn}. \fi
The Leonard system
$\Phi$ is said to have  {\em dual  $q$-Hahn type}
whenever 
there exist scalars $q, h, h^*, s, r$ such that 
\begin{itemize}
\item[\rm (i)]
each of $q, h, h^*, s, r$ is nonzero;
\item[\rm (ii)]
none of $q^i, r q^i, s q^i/r$ is equal to $1$ for $1 \leq i \leq d$;
\item[\rm (iii)]
$s q^i \neq 1$ for $2 \leq i \leq 2d$;
\item[\rm (iv)]
for $0 \leq i \leq d$,
\begin{align*}
\theta_i &= \theta_0+h(1-q^i)(1-sq^{i+1})q^{-i},
\\
\theta^*_i &=\theta^*_0+h^*(1-q^i)q^{-i};
\end{align*}
\item[\rm (v)]
for $1 \leq i \leq d$,
\begin{align*}
\varphi_i &= hh^*q^{1-2i}(1-q^i)(1-q^{i-d-1})(1-rq^i),
\\
\phi_i &= hh^*q^{d+2-2i}(1-q^i)(1-q^{i-d-1})(s-rq^{i-d-1}).
\end{align*}
\end{itemize}
\end{defi}

\begin{defi}   {\rm  (See \cite[Example 5.6]{T:parray}.) }
\label{def:quantumqKrawt}    \samepage
\ifDRAFT {\rm def:quantumqKrawt}. \fi
The Leonard system
$\Phi$ is said to have  {\em quantum  $q$-Krawtchouk type}
whenever 
there exist scalars $q, h^*, s, r$ such that 
\begin{itemize}
\item[\rm (i)]
each of $q, h^*, s,r$  is nonzero;
\item[\rm (ii)]
$q^i \neq 1$, $s q^i/r \neq 1$ for $1 \leq i \leq d$;
\item[\rm (iii)]
for $0 \leq i \leq d$,
\begin{align*}
\theta_i &= \theta_0-sq(1-q^i),
\\
\theta^*_i &= \theta^*_0+h^*(1-q^i)q^{-i};
\end{align*}
\item[\rm (iv)]
for  $1 \leq i \leq d$,
\begin{align*}
\varphi_i &= -rh^*q^{1-i}(1-q^i)(1-q^{i-d-1}),
\\
\phi_i &= h^*q^{d+2-2i}(1-q^i)(1-q^{i-d-1})(s-rq^{i-d-1}).
\end{align*}
\end{itemize}
\end{defi}

\begin{defi}              {\rm  (See \cite[Example 5.7]{T:parray}.) }
\label{def:qKrawt}    \samepage
\ifDRAFT {\rm def:qKrawt}. \fi
The Leonard system
$\Phi$ is said to have  {\em  $q$-Krawtchouk type}
whenever 
there exist scalars $q, h, h^*, s^*$ such that 
\begin{itemize}
\item[\rm (i)]
each of $q,h, h^*, s^*$ is nonzero;
\item[\rm (ii)]
$q^i \neq 1$ for $1 \leq i \leq d$;
\item[\rm (iii)]
 $s^* q^i  \neq 1$
for $2 \leq i \leq 2d$;
\item[\rm(iv)]
for $0 \le i \leq d$,
\begin{align*}
\theta_i &= 
\theta_0+h(1-q^i)q^{-i},
\\
\theta^*_i &=  \theta^*_0+h^*(1-q^i)(1-s^*q^{i+1})q^{-i};
\end{align*}
\item[\rm (v)]
for $1 \leq i \leq d$,
\begin{align*}
\varphi_i &= h h^*q^{1-2i}(1-q^i)(1-q^{i-d-1}),
\\
\phi_i &= hh^*s^*q(1-q^i)(1-q^{i-d-1}).
\end{align*}
\end{itemize}
\end{defi}

\begin{defi}          {\rm  (See \cite[Example 5.8]{T:parray}.) }
\label{ded:affineqKrawt}   \samepage
\ifDRAFT {\rm def:affineqKrawt}. \fi
The Leonard system
$\Phi$ is said to have  {\em affine  $q$-Krawtchouk type}
whenever 
there exist scalars $q, h, h^*, r$ such that 
\begin{itemize}
\item[\rm (i)]
each of $q, h, h^*, r$ is nonzero;
\item[\rm (ii)]
$q^i \neq 1$, $r q^i \neq 1$ for $1 \leq i \leq d$;
\item[\rm (iii)]
for $0 \leq i \leq d$,
\begin{align*}
\theta_i &= \theta_0+h(1-q^i)q^{-i},
\\
\theta^*_i &= \theta^*_0+h^*(1-q^i)q^{-i};
\end{align*}
\item[\rm (iv)]
for 
 $1 \leq i \leq d$,
\begin{align*}
\varphi_i &= h h^*q^{1-2i}(1-q^i)(1-q^{i-d-1})(1-r q^i),
\\
\phi_i &= -h h^*r q^{1-i}(1-q^i)(1-q^{i-d-1}).
\end{align*}
\end{itemize}
\end{defi}

\begin{defi}        {\rm  (See \cite[Example 5.9]{T:parray}.) }
\label{def;dualqKrawt}  \samepage
\ifDRAFT {\rm de:dualqKrawt}. \fi
The Leonard system
$\Phi$ is said to have  {\em dual   $q$-Krawtchouk type}
whenever 
there exist scalars $q, h, h^*, s$ such that 
\begin{itemize}
\item[\rm (i)]
each of $q,h,h^*, s$ is nonzero;
\item[\rm (ii)]
$q^i \neq 1$ for $1 \leq i \leq d$;
\item[\rm (iii)]
$s q^i \neq 1$ for $2 \leq i \leq 2d$;
\item[\rm (iv)]
for $0 \leq i \leq d$,
\begin{align*}
\theta_i &= \theta_0+h(1-q^i)(1-sq^{i+1})q^{-i},
\\
\theta^*_i &= \theta^*_0+h^*(1-q^i)q^{-i};
\end{align*}
\item[\rm (v)]
for $1 \leq i \leq d$,
\begin{align*}
\varphi_i &= hh^*q^{1-2i}(1-q^i)(1-q^{i-d-1}),
\\
\phi_i &= hh^*sq^{d+2-2i}(1-q^i)(1-q^{i-d-1}).
\end{align*}
\end{itemize}
\end{defi}

\begin{defi}   {\rm (See \cite[Example 5.10]{T:parray}.) }
\label{def:Racah}   \samepage
\ifDRAFT {\rm  def:Racah}. \fi
The Leonard system
$\Phi$ is said to have {\em Racah type} whenever
there exist scalars $h, h^*, s, s^*, r_1, r_2$ such that 
\begin{itemize}
\item[\rm (i)]
each of $h, h^*$ is nonzero;
\item[\rm (ii)]
$r_1 + r_2 = s + s^* + d + 1$;
\item[\rm (iii)]
the characteristic of $\F$ is $0$ or a prime greater than $d$;
\item[\rm (iv)]
none of $r_1, r_2, s^* - r_1, s^* - r_2$ is equal to $-i$ for $1 \leq i \leq d$;
\item[\rm (v)]
$s \neq -i$, $s^* \neq -i$ for $2 \leq i \leq 2d$;
\item[\rm (vi)]
for $0 \leq i \leq d$,
\begin{align*}
\th_i &= \th_0 + h  i (i+1+ s),
\\
\th^*_i &= \th^*_0 + h^* i (i+1+ s^*);
\end{align*}
\item[\rm (vii)]
for $1 \leq i  \leq d$,
\begin{align*}
\vphi_i &= h h^* i (i-d-1) (i+r_1)(i+r_2),
\\
\phi_i &= h h^* i (i-d-1)(i+s^* - r_1)(i+s^* - r_2).
\end{align*}
\end{itemize}
\end{defi}

\begin{defi}   {\rm (See \cite[Example 5.11]{T:parray}.) }
\label{def:Hahn}  \samepage
\ifDRAFT {\rm def:Hahn}. \fi
The Leonard system
$\Phi$ is said to have {\em Hahn type} whenever
there exist scalars $h^*,  s, s^*, r$ such that 
\begin{itemize}
\item[\rm (i)]
each of $h^*, s$ is nonzero;
\item[\rm (ii)]
the characteristic of $\F$ is $0$ or a prime greater than $d$;
\item[\rm (iii)]
neither of $r, s^* - r$ is equal to $-i$ for $1 \leq i \leq d$;
\item[\rm (iv)]
$s^* \neq -i$ for $2 \leq i \leq 2d$;
\item[\rm (v)]  
for $0 \leq i \leq d$,
\begin{align*}
\theta_i &= \theta_0+s i,
\\
\theta^*_i &= \theta^*_0+h^*i(i+1+s^*);
\end{align*}
\item[\rm (vi)]
for $1 \leq i \leq d$,
\begin{align*}
\varphi_i &= h^*si(i-d-1)(i+r),
\\
\phi_i &= -h^*si(i-d-1)(i+s^*-r).
\end{align*}
\end{itemize}
\end{defi}

\begin{defi}    {\rm (See \cite[Example 5.12]{T:parray}.) }
\label{def:dualHahn}    \samepage
\ifDRAFT {\rm def:dualHahn}. \fi
The Leonard system
$\Phi$ is said to have {\em dual Hahn type} whenever
there exist scalars $h,  s, s^*, r$ such that
\begin{itemize}
\item[\rm (i)]
each of $h, s^*$  is nonzero;
\item[\rm (ii)]
the characteristic of $\F$ is $0$ or a prime greater than $d$;
\item[\rm (iii)]
neither of $ r, s-r$ is equal to $-i$ for $1 \leq i \leq d$;
\item[\rm (iv)]
$s \neq -i$ for $2 \leq i \leq 2d$;
\item[\rm (v)]
for $0 \leq i \leq d$,
\begin{align*}
\theta_i &= \theta_0+hi(i+1+s),
\\
\theta^*_i &= \theta^*_0+s^*i;
\end{align*}
\item[\rm (vi)]
for $1 \leq i \leq d$,
\begin{align*}
\varphi_i &= hs^*i(i-d-1)(i+r),
\\
\phi_i &= hs^*i(i-d-1)(i+r-s-d-1).
\end{align*}
\end{itemize}
\end{defi}

\begin{defi}   {\rm (See \cite[Example 5.13]{T:parray}.) }
\label{def:Krawt} \samepage
\ifDRAFT {\rm def:Krawt}. \fi
The Leonard system
$\Phi$ is said to have {\em Krawtchouk type} whenever
there exist scalars $s, s^*, r$ such that
\begin{itemize}
\item[\rm (i)]
each of $s,s^*, r$ is nonzero;
\item[(ii)]
the characteristic of $\F$ is $0$ or a prime greater than $d$;
\item[\rm (iii)]
$r \neq s s^*$;
\item[\rm (iv)]
for $0 \leq i \leq d$,
\begin{align*}
\th_i &= \th_0 + s i,
\\
\th^*_i &= \th^*_0 +s^* i;
\end{align*}
\item[\rm (v)]
for $1 \leq i  \leq d$,
\begin{align*}
\vphi_i &=  r i (i-d-1),
\\
\phi_i &= (r-s s^*) i (i-d-1).
\end{align*}
\end{itemize}
\end{defi}

\begin{defi}   {\rm (See \cite[Example 5.14]{T:parray}.) }
\label{def:BI} \samepage
\ifDRAFT {\rm def:BI}. \fi
The Leonard system
$\Phi$ is said to have {\em Bannai/Ito type} whenever
there exist scalars $h, h^*, s, s^*, r_1, r_2$ such that
\begin{itemize}
\item[\rm (i)]
each of $h, h^*$ is nonzero;
\item[\rm (ii)]
the characteristic of $\F$ is $0$ or an odd prime greater than $d/2$;
\item[\rm (iii)]
$r_1 + r_2 = -s - s^* +d+1$;
\item[\rm (iv)]
neither of $r_1$, $-s^* - r_1$ is equal to $-i$  for $1 \leq i \leq d$, $d-i$   even;
\item[\rm (v)]
neither of $r_2$, $-s^* - r_2$ is equal to $-i$  for $1 \leq i \leq d$,  $i$  odd;
\item[\rm (vi)]
neither of $s$, $s^*$ is equal to $2i$ for $1 \leq i \leq d$;
\item[\rm (vii)]
for $0 \leq i  \leq d$,
\begin{align*}
\th_i &= \th_0 + h \big( s-1+(1-s+2i) (-1)^{i} \big),
\\
\th^*_i &= \th^*_0 + h^* \big( s^* - 1 + (1 - s^* + 2i) (-1)^{i} \big);
\end{align*}
\item[\rm (viii)]
for $1 \leq i  \leq d$,
\begin{align*}
\vphi_i &= 
\begin{cases}
- 4 h h^* i (i+r_1) &  \text{ if $i$ even, $d$ even},
\\
- 4 h h^* (i-d-1)(i+r_2) &  \text{ if $i$ odd, $d$ even},
\\
- 4 h h^* i (i-d-1) &   \text{ if $i$ even, $d$ odd},
\\
-4 h h^* (i+r_1)(i+r_2) &   \text{ if $i$ odd, $d$ odd},
\end{cases}
\\
\phi_i &=
\begin{cases}
4 h h^* i (i - s^* - r_1) &  \text{ if $i$ even, $d$ even},
\\
4 h h^* (i-d-1)(i - s^* - r_2) &  \text{ if $i$ odd, $d$ even},
\\
- 4 h h^* i (i-d-1) &   \text{ if $i$ even, $d$ odd},
\\
-4 h h^* (i - s^* - r_1)(i - s^* - r_2) &   \text{ if $i$ odd, $d$ odd}.
\end{cases}
\end{align*}
\end{itemize}
\end{defi}

\begin{defi}   {\rm (See \cite[Example 5.15]{T:parray}.) }
\label{def:Orphan} \samepage
\ifDRAFT {\rm def:Orphan}. \fi
The Leonard system
$\Phi$ is said to have {\em Orphan type} whenever
$d=3$ and the characteristic of $\F$ is $2$,
and there exist scalars $h, h^*, s, s^*,r$ such that
\begin{itemize}
\item[\rm (i)]
each of $h,h^*,s,s^*, r$ is nonzero;
\item[\rm (ii)]
$s \neq 1$, $s^* \neq 1$;
\item[\rm (iii)]
$r$ is equal to none of $s+s^*, s(1+s^*), s^*(1+s)$;
\item[(iv)]
\begin{align*}
\th_1 &= \th_0 + h(s+1),
&
\th_2 &= \th_0 + h,
&
\th_3 &= \th_0 + h s,
\\
\th^*_1 &= \th^*_0 + h^*(s^*+1),
&
\th^*_2 &= \th^*_0 + h^*,
&
\th^*_3 &= \th^*_0 + h^* s^*;
\end{align*}
\item[\rm (v)]
\begin{align*}
\vphi_1 &= h h^* r,
&
\vphi_2 &= h h^*,
&
\vphi_3 &= h h^* ( r + s + s^*),
\\
\phi_1 &= h h^*( r + s + s s^*),
&
\phi_2 &= h h^*,
&
\phi_3 &= h h^* (r + s^* + s s^*).
\end{align*}
\end{itemize}
\end{defi}

\section{About the equation Proposition \ref{prop:cond}(vi)}
\label{sec:neweq}
\ifDRAFT {\rm sec:neweq}. \fi

We continue to discuss the Leonard system $\Phi$ from \eqref{eq:Phi}.
By Proposition \ref{prop:cond}, we have $ {\mathcal Z}(\Phi) \neq 0$ if and only if
the following equation holds for $1 \leq i,j \leq d-1$:
\begin{equation}  \label{eq:pi2}
(a_i - a_0)
(\th^*_i - \th^*_d)
(a_j - a_d)
(\th^*_j - \th^*_0)
=
(a_i -a_d)
(\th^*_i - \th^*_0)
(a_j - a_0)
(\th^*_j - \th^*_d).
\end{equation}

In this section, we evaluate the left-hand side of \eqref{eq:pi2} minus
the right-hand side of \eqref{eq:pi2}.
We use the following expression:
\begin{equation}    \label{eq:Q}
\frac{
(\th^*_0 - \th^*_i)
(\th^*_0 - \th^*_j)
(\th^*_0 - \th^*_d)
(\th^*_i - \th^*_j)
(\th^*_i - \th^*_d)
(\th^*_j - \th^*_d)
}
{
(\th^*_0 - \th^*_1)
(\th^*_{i-1} - \th^*_i)
(\th^*_i - \th^*_{i+1})
(\th^*_{j-1}- \th^*_j)
(\th^*_j - \th^*_{j+1})
(\th^*_{d-1} - \th^*_d)
}.
\end{equation}

\begin{prop}   \label{prop:pi2}    \samepage
\ifDRAFT {\rm prop:pi2}. \fi
For $1 \leq i,j  \leq d-1$ 
the  left-hand side of \eqref{eq:pi2} minus the right-hand side of \eqref{eq:pi2} is
equal to \eqref{eq:Q} times the factor given in the table below:
\[
\begin{array}{c|c}
\text{Type} & \text{Factor}
\\ \hline
\text{$q$-Racah}                     \rule{0mm}{3.5ex}
&
  h^2 h^{*2}
 q^{-3-d}
 (q-1)^4
 (q^2-1)^2 
 (s^* - r_1^2)
(s^* - r_2^2)/s^*
\\
\text{$q$-Hahn }                  \rule{0mm}{3.2ex}
&
h^2 h^{*2} q^{-3-d}(q-1)^4(q^2-1)^2 (s^* - r^2)
\\
\text{ dual  $q$-Hahn }              \rule{0mm}{3.2ex}
&
- h^2 h^{*2} q^{-3-d}(q-1)^4(q^2-1)^2 r^2
\\
\text{ quantum $q$-Krawtchouk}               \rule{0mm}{3.2ex}
&
- h^{*2} q^{-3-d}(q-1)^4(q^2-1)^2 r^2
\\
\text{ $q$-Krawtchouk}                \rule{0mm}{3.2ex}
&
h^2 h^{*2} q^{-3-d}(q-1)^4(q^2-1)^2 s^*
\\
\text{affine  $q$-Krawtchouk}                \rule{0mm}{3.2ex}
&
- h^2 h^{*2} q^{-3-d} (q-1)^4 (q^2-1)^2 r^2
\\
\text{dual   $q$-Krawtchouk}                \rule{0mm}{3.2ex}
&
0
\\
\text{Racah}                \rule{0mm}{3.2ex}
&
4 h^2 h^{*2} (s^* - 2 r_1)(s^* - 2 r_2)
\\
\text{Hahn}                \rule{0mm}{3.2ex}
&
0
\\
\text{dual Hahn}                \rule{0mm}{3.2ex}
&
-4 h^2 s^{*2}
\\
\text{Krawtchouk}                \rule{0mm}{3.2ex}
&
0
\\
\text{Bannai/Ito}                \rule{0mm}{3.2ex}
&
 64 (-1)^{d+1} h^2 h^{* 2}  (s^* + 2 r_1)(s^* + 2 r_2)
\\
\text{Orphan}                \rule{0mm}{3.2ex}
&
h^2 h^{*2} (s^{*2} + 1)
\end{array}
\]
\end{prop}

\begin{proof}
Routine verification using the data in Section \ref{sec:types}.
\end{proof}

\section{A necessary and sufficient condition for $ {\mathcal Z}(\Phi) \neq 0$}
\label{sec:main}
\ifDRAFT {\rm sec:main}. \fi

We continue to discuss the Leonard system $\Phi$ from \eqref{eq:Phi}.
In this section, we find a necessary and sufficient condition for
${\mathcal Z}(\Phi) \neq 0$.

\begin{theorem}   \label{thm:main1}    \samepage
\ifDRAFT {\rm thm:main1}. \fi
Assume that $\Phi$ has one of the following types:
\begin{center}   \samepage
dual $q$-Krawtchouk,
\quad
Hahn,
\quad
Krawtchouk.
\end{center}
Then $ {\mathcal Z}(\Phi) \neq 0$.
\end{theorem}

\begin{proof}
By Proposition \ref{prop:cond} and 
Proposition \ref{prop:pi2}.
\end{proof}

\newpage

\begin{theorem}   \label{thm:main2}   
\ifDRAFT {\rm thm:main2}. \fi
Assume that $\Phi$ has one of the following types:
\begin{center}   \samepage
dual $q$-Hahn,
\quad
quantum $q$-Krawtchouk,
\quad
$q$-Krawtchouk,
\\
affine $q$-Krawtchouk,
\quad
dual Hahn,
\quad
Orphan.
\end{center}
Then $ {\mathcal Z}(\Phi) = 0$.
\end{theorem}

\begin{proof}
By Proposition \ref{prop:cond} and 
Proposition \ref{prop:pi2}.
\end{proof}

\begin{theorem}   \label{thm:main}    \samepage
\ifDRAFT {\rm thm:main}. \fi
For the following types, $ {\mathcal Z}(\Phi) \neq 0$
if and only if the given condition is satisfied.
\begin{center}
\begin{tabular}{c|c}
{\rm Type} & {\rm Condition}
\\ \hline
$q$-Racah
& 
$s^* = r_1^2$ 
$\quad$ or $\quad$
$s^* = r_2^2$    \rule{0mm}{3ex}
\\
$q$-Hahn
& 
$s^* = r^2$                    \rule{0mm}{2.8ex}
\\
Racah
&
$\;$
$s^* =2 r_1$ 
$\quad$ or $\quad$
$s^* = 2 r_2$           \rule{0mm}{2.8ex}
\\
Bannai/Ito
&
$\;$
$s^* = - 2 r_1$ 
$\quad$ or $\quad$
$s^* = -2 r_2$                             \rule{0mm}{2.8ex}
\end{tabular}
\end{center}
\end{theorem}

\begin{proof}
By Proposition \ref{prop:cond} and 
Proposition \ref{prop:pi2}.
\end{proof}

\section{The Leonard systems with $\dim {\mathcal Z}(\Phi) = 2$}
\label{sec:dimZ2}
\ifDRAFT {\rm sec:dimZ2}. \fi

We continue to discuss the Leonard system $\Phi$ from \eqref{eq:Phi}.
Recall by Lemma \ref{lem:condd} that 
$\dim {\mathcal Z}(\Phi) = 2$ if and only if 
$a_0 = a_1 = \cdots = a_d$.
In this section,  we examine this condition
for each of the types listed in Section \ref{sec:types}.

\begin{prop}   {\rm (See \cite[Section 17]{NT:Nbip}.) }
\label{prop:Bip}     \samepage
\ifDRAFT {\rm prop:Bip}. \fi
If $\dim {\mathcal Z}(\Phi) = 2$ then $\Phi$ has one of the types:
\begin{center}
$q$-Racah,
\quad
dual $q$-Krawtchouk,
\quad
Hahn,
\quad
Krawtchouk,
\quad
Bannai/Ito with $d$ even.
\end{center}
For each of the above types,
$\dim {\mathcal Z}(\Phi) = 2$
if and only if the following condition holds:
\begin{center}
\begin{tabular}{c|c}
{\em Type} & {\rm Condition}
\\ \hline
$q$-Racah
&
$\quad$
$ s^* = r_1^2\;\;$  and  $\;\;s=-q^{-d-1}$                \rule{0mm}{3.5ex}
\\
dual $q$-Krawtchouk
&
$s= - q^{-d-1}$             \rule{0mm}{3.2ex}
\\
Hahn
&
$s^* = 2 r$               \rule{0mm}{3.2ex}
\\
Krawtchouk 
&
$ s s^* = 2 r$                  \rule{0mm}{3.2ex}
\\
Bannai/Ito,  $d$ even
& \;\; $s^* = - 2 r_1 \;\;$ and  $\;\;s=d+1$                  \rule{0mm}{3.2ex}
\end{tabular}
\end{center}
\end{prop}

\section{A relation between $a^-_i$ and $a^+_i$}
\label{sec:rel}
\ifDRAFT {\rm sec:rel}. \fi

We continue to  discuss the Leonard system $\Phi$ from \eqref{eq:Phi}.
Recall  the scalars $\{a^-_i\}_{i=0}^d$ and $\{a^+_i\}_{i=0}^d$  from Definition \ref{def:amap}.
In this section 
we assume ${\mathcal Z}(\Phi) \neq 0$,
and obtain a relation between $a^-_i$ and $a^+_i$ for $0 \leq i \leq d$.

\begin{prop}    \label{prop:amap}    \samepage
\ifDRAFT {\rm prop:smsp}. \fi
Assume that  $ {\mathcal Z}(\Phi) \neq 0$.
Then
$a^-_i$ and $a^+_i$ are related as follows  for $0 \leq i \leq d$:
\[
\begin{array}{c|c}
\text{Case} & \text{Relation between $a^-_i$ and  $a^+_i$}
\\ \hline
\text{ $q$-Racah},       \rule{0mm}{3.5ex}
s^* = r_1^2
&
a^-_i
q^d
(r_1 + 1)
(r_1 q + 1)
= a^+_i
 (r_1 q^d + 1)
(r_1 q^{d+1} + 1)
\\
\text{ $q$-Racah},       \rule{0mm}{3.2ex}
s^* = r_2^2
&
a^-_i
q^d
(r_2 + 1)
(r_2 q + 1)
= a^+_i
 (r_2 q^d + 1)
(r_2 q^{d+1} + 1)
\\
\text{ $q$-Hahn},    \;    \rule{0mm}{3.5ex}
s^* = r^2
&
 a^-_i 
q^d
(r+1)
(r q + 1)
= a^+_i
(r q^d+1)
(r q^{d+1} + 1)
\\
\text{dual $q$-Krawtchouk}    \;    \rule{0mm}{3.5ex}
&
a^-_i q^d  = a^+_i
\\
\text{Racah}, s^* = 2 r_1    \;    \rule{0mm}{3.5ex}
&
a^-_i = a^+_i
\\
\text{Racah}, s^* = 2 r_2    \;    \rule{0mm}{3.5ex}
&
a^-_i = a^+_i
\\
\text{Hahn}    \;    \rule{0mm}{3.5ex}
&
a^-_i 
s^*
(s^* + 2)
=
a^+_i
(s^* + 2d)
(s^* + 2d +2)
\\
\text{Krawtchouk}    \;    \rule{0mm}{3.5ex}
&
 a^-_i = a^+_i
\\
\text{Bannai/Ito, $d$ even},    \;    \rule{0mm}{3.5ex}
s^* = - 2 r_1
&
a^-_i (r_1+1) = a^+_i (r_1 + d + 1)
\\
\text{Bannai/Ito, $d$ even},    \;    \rule{0mm}{3.5ex}
s^* = - 2 r_2
&
 a^-_i r_2 =  a^+_i (r_2 +d)
\\
\text{Bannai/Ito, $d$ odd},    \;    \rule{0mm}{3.5ex}
s^* = - 2 r_1
&
 a^-_i r_1  = - a^+_i (r_1 + d + 1)
\\
\text{Bannai/Ito, $d$ odd},    \;    \rule{0mm}{3.5ex}
s^* = - 2 r_2
&
 a^-_i r_2 = - a^+_i (r_2 +d+1)
\end{array}
\]
\end{prop}

\begin{proof}
Routine verification using Lemma \ref{lem:aiformula} 
and the data in Section \ref{sec:types}.
\end{proof}

\section{The  Leonard systems $\Phi$ such that $ \dim {\mathcal Z}(\Phi)=1$}
\label{sec:rev} 
\ifDRAFT {\rm sec:rev}. \fi

We continue to discuss the Leonard system $\Phi$ from \eqref{eq:Phi}.
In Lemma \ref{lem:EB2} we gave a basis for 
$ {\mathcal Z}(\Phi)$
under the assumption that $\dim {\mathcal Z}(\Phi)= 2$.
In this section, we give a basis for
$ {\mathcal Z}(\Phi)$
under the assumption that $\dim {\mathcal Z}(\Phi)= 1$.

\begin{prop} \label{prop:amap2}     \samepage
\ifDRAFT {\rm prop:amap2}. \fi
Assume that  ${\mathcal Z}(\Phi) \neq 0$.
Then the following  element is nonzero and  contained in  
${\mathcal Z}(\Phi)$.
\[
\begin{array}{c|l}
\text{Case}
& \qquad\quad
\text{A nonzero element in ${\mathcal Z}(\Phi)$}
\\ \hline
\
\text{$q$-Racah, $s^* =  r_1^2$}            \rule{0mm}{5ex}
&
\begin{array}{l}
(A - a_0 I)(A^* - \th^*_d I) q^d  (r_1 + 1)
(r_1 q + 1)
\\ \qquad
-
(A - a_d I)(A^* - \th^*_0 I)  (r_1 q^d + 1)
(r_1 q^{d+1} + 1)
\end{array}
\\
\text{$q$-Racah, $s^* =  r_2^2$}            \rule{0mm}{5ex}
&
\begin{array}{l}
(A - a_0 I)(A^* - \th^*_d I) q^d
(r_2 + 1)
(r_2 q + 1)
\\ \ \qquad
-
(A - a_d I)(A^* - \th^*_0 I)  (r_2 q^d + 1)
(r_2 q^{d+1} + 1)
\end{array}
\\
\text{$q$-Hahn,  $s^* = r^2$  }            \rule{0mm}{5ex}
&
\begin{array}{l}
(A - a_0 I)(A^* - \th^*_d I)
q^d
(r+1)
(r q + 1)
\\ \qquad
-
(A - a_d I)(A^* - \th^*_0 I) (r q^d+1)
(r q^{d+1} + 1)
\end{array}
\\
\text{dual $q$-Krawtchouk}            \rule{0mm}{3.5ex}
&
\begin{array}{l}
(A - a_0 I)(A^* - \th^*_d I ) q^d
-
(A - a_d I)(A^* - \th^*_0 I) 
\end{array}
\\
\text{Racah, $s^* = 2 r_1$ }             \rule{0mm}{3.5ex}
&
\begin{array}{l}
(A - a_0 I)(A^* - \th^*_d I)
-
(A - a_d I)(A^* - \th^*_0 I)
\end{array}
\\
\text{Racah,  $s^* = 2 r_2$}             \rule{0mm}{3.5ex}
&
\begin{array}{l}
(A - a_0 I)(A^* - \th^*_d I)
-
(A - a_d I)(A^* - \th^*_0 I)
\end{array}
\\
\text{Hahn}          \rule{0mm}{5ex}
&
\begin{array}{l}
(A - a_0 I)(A^* - \th^*_d I)s^*
(s^* + 2)
\\ \qquad
-
(A - a_d I)(A^* - \th^*_0 I) (s^* + 2d)
(s^* + 2d +2)
\end{array}
\\
\text{Krawtchouk}            \rule{0mm}{3.5ex}
& 
\begin{array}{l}
(A - a_0 I)(A^* - \th^*_d I )
-
(A - a_d I)(A^* - \th^*_0 I)
\end{array}
\\
\text{Bannai/Ito, $d$  even, $s^* = - 2 r_1$}            \rule{0mm}{5ex}
& 
\begin{array}{l}
(A - a_0 I)(A^* - \th^*_d I)(r_1+1) 
\\ \qquad
-
(A - a_d I)(A^* - \th^*_0 I) (r_1 + d + 1)
\end{array}
\\
\text{Bannai/Ito, $d$ even, $s^* = - 2 r_2$}            \rule{0mm}{5ex}
&
\begin{array}{l}
(A - a_0 I)(A^* - \th^*_d I)  r_2 
\\ \qquad
-
(A - a_d I)(A^* - \th^*_0 I)  (r_2 +d)
\end{array}
\\
\text{Bannai/Ito, $d$ odd, $s^* = - 2 r_1$}            \rule{0mm}{5ex}
& 
\begin{array}{l}
(A - a_0 I)(A^* - \th^*_d I)  r_1
\\ \qquad
+
(A - a_d I)(A^* - \th^*_0 I)  (r_1 +d+1)
\end{array}
\\
\text{Bannai/Ito, $d$ odd, $s^* = - 2 r_2$}            \rule{0mm}{5ex}
&
\begin{array}{l}
(A - a_0 I)(A^* - \th^*_d I)  r_2 
\\ \qquad
+
(A - a_d I)(A^* - \th^*_0 I)  (r_2 +d+1)
\end{array}
\end{array}
\]
Moreover, the above element is a basis of
${\mathcal Z}(\Phi)$,
provided that $\dim  {\mathcal Z}(\Phi) =  1$.
\end{prop}

\begin{proof}
The given element is contained in  ${\mathcal Z}(\Phi)$
by Lemma \ref{lem:cond}, Definition \ref{def:amap}, and Proposition \ref{prop:amap}.
The given element is nonzero, by Lemma \ref{lem:linindep} and the inequalities
in Section \ref{sec:types}.
The result follows.
\end{proof}

\section{Self-dual  Leonard pairs and Leonard systems}
\label{sec:SD}
\ifDRAFT {\rm sec:SD}. \fi

We are done with our first topic. We now
consider our second topic, concerning
spin Leonard pairs.
These Leonard pairs have a property
called self-dual. In this section, we describe
the self-dual property.

\begin{defi}    {\rm (See \cite[Definition 7.1]{Cur2}, \cite[Definition 3.4]{NT:sd}.) }
\label{def:sdLP}    \samepage
\ifDRAFT {\rm lem:sdLP}. \fi
A Leonard pair $A,A^*$ on $V$ is said to be {\em self-dual}
whenever $A, A^*$ is isomorphic to the Leonard pair $A^*, A$.
\end{defi}

\begin{defi}       {\rm (See \cite[Definition 3.4]{NT:sd}.) }
\label{def:sd}    \samepage
\ifDRAFT {\rm def:sd}. \fi
A Leonard system $\Phi$ on $V$ is said to be {\em self-dual}
whenever $\Phi$ is isomorphic to $\Phi^*$.
\end{defi}
  
\begin{lemma}    \label{lem:sdLP}    \samepage
\ifDRAFT {\rm lem:sdLP}. \fi
Let $A,A^*$ denote a self-dual Leonard pair on $V$.
Then there exists an associated Leonard system $\Phi$ that is
self-dual.
\end{lemma}

\begin{proof}
Let $\sigma$ denote an isomorphism of Leonard pairs from  $A,A^*$ to $A^*, A$.
Let $\{E_i\}_{i=0}^d$ denote a standard ordering of the primitive idempotents of $A$.
For $0 \leq i \leq d$ define $E^*_i = E^\sigma_i$.
Then
$\Phi = (A: \{E_i\}_{i=0}^d; A^* ; \{E^*_i\}_{i=0}^d)$
is a Leonard system that is associated to $A,A^*$.
By construction $\Phi^\sigma = \Phi^*$, so $\Phi$ is isomorphic to $\Phi^*$.
Thus $\Phi$ is self-dual.
\end{proof}

\begin{lemma}   {\rm (See \cite[Lemma 7.4]{NT:sd}.) }
\label{lem:sd}    \samepage
\ifDRAFT {\rm lem:sd}. \fi
Recall the Leonard system $\Phi$ from \eqref{eq:Phi}.
Then $\Phi$ is self-dual if and only if
\[
\th_i = \th^*_i
\qquad\qquad\qquad
(0 \leq i \leq d).
\]
If this case,
\[
\phi_i = \phi_{d-i+1}
\qquad\qquad
(1 \leq i \leq d).
\]
\end{lemma}

\begin{lemma}   \label{lem:typesd}   \samepage
\ifDRAFT {\rm lem:typesd}. \fi
Recall the Leonard system $\Phi$ from \eqref{eq:Phi}.
If $\Phi$ is self-dual, then $\Phi$ has one of the following types:
\begin{center}
$q$-Racah,
\quad
affine $q$-Krawtchouk,
\quad
Racah,
\quad
Krawtchouk,
\quad
Bannai/Ito,
\quad
Orphan.
\end{center}
For each of the above types,  $\Phi$ is self-dual if and only if $\th_0 =\th^*_0$ and 
the following condition holds:
\begin{center}
\begin{tabular}{c|c}
 Type  &  Condition
\\ \hline
$q$-Racah &  \quad $h=h^* \quad$  and $\quad s=s^*$    \rule{0mm}{3.5ex}
\\
affine $q$-Krawtchouk                       \rule{0mm}{3.2ex}
&
\quad $h=h^*$
\\
Racah & \quad $h=h^*$\quad and \quad  $ s=s^*$  \rule{0mm}{3.2ex}
\\
Krawtchouk &  $ s=s^*$  \rule{0mm}{3.2ex}
\\
Bannai/Ito & $\quad h = h^*$ \quad and  \quad $s=s^*$  \rule{0mm}{3.2ex}
\\
Orphan &  $\quad h = h^*$ \quad and  \quad $s=s^*$  \rule{0mm}{3.2ex}
\end{tabular}
\end{center}
\end{lemma}
\begin{proof}
Routine verification 
using Lemma \ref{lem:sd}
and the data in Section \ref{sec:types}.
\end{proof}

\section{Spin Leonard pairs}
\label{sec:spinLP}
\ifDRAFT {\rm sec:spinLP}> \fi

In this section, we recall  the notion of a spin Leonard pair.
We then characterize the spin Leonard pairs using the  zero diagonal space.

\begin{defi}  {\rm (See \cite[Definition 1.2]{Cur}.) }
\label{def:spinLP}   \samepage
\ifDRAFT {\rm def:spinLP}. \fi
Let $A,A^*$ denote a Leonard pair on $V$.
This Leonard pair is said to have {\em spin} 
whenever there exist invertible $W, W^* \in \text{\rm End}(V)$
such  that
\begin{align*}
 W A & = A W, 
\\
 W^* A^* &= A^* W^*,  
\\
W A^* W^{-1} &= (W^*)^{-1} A W^*.  
\end{align*}
\end{defi}

\begin{lemma}   {\rm (See \cite[Theorem 1.6]{Cur}, \cite[Lemma  9.2(iii)]{Cur2}.) }
\label{lem:SDpre}   \samepage
\ifDRAFT {\rm lem:SDpre}. \fi
Let $A,A^*$ denote a spin Leonard pair on $V$.
Then  $A,A^*$ is self-dual.
\end{lemma}

\begin{prop}   {\rm (See \cite[Lemmas 1.7--1.11, 1.13]{Cur}.) }
\label{prop:spinLP}    \samepage
\ifDRAFT {\rm prop:spinLP}. \fi
Recall the Leonard system $\Phi$ from \eqref{eq:Phi},
and assume that $\Phi$ is self-dual.
If $A,A^*$ has spin,
then $\Phi$ has one of the following types:
\begin{center}
$q$-Racah,
\quad
Racah,
\quad
Krawtchouk,
\quad
Bannai/Ito.
\end{center}
If $\Phi$ has Krawtchouk type, then $A,A^*$ has spin.
For the remaining types, $A,A^*$ has spin if and only if
the following condition holds:
\begin{center}
\begin{tabular} {c|c}
Type & Condition
\\ \hline
$q$-Racah
&
$s=r_1^2$ \quad or \quad $s=r_2^2$                   \rule{0mm}{3.5ex}
\\
Racah
&
$s= 2 r_1$ \quad or \quad $s=2 r_2$                 \rule{0mm}{3.2ex}
\\ 
Bannai/Ito
&
\quad $s= -2 r_1$ \quad or \quad $s= - 2 r_2$              \rule{0mm}{3.2ex}
\end{tabular}
\end{center}
\end{prop}

\begin{note}   \label{note:typo}   \samepage
\ifDRAFT {\rm note:typo}. \fi
In \cite[Lemma 1.11]{Cur},
there is a typo.
In the formula for $b_i$ ($i$ odd),
the minus sign should be removed.
\end{note}

\begin{theorem}   \label{thm:characterize}   \samepage
\ifDRAFT {\rm thm:characterize}. \fi
Let $A,A^*$ denote a Leonard pair on $V$.
Then the following are equivalent:
\begin{itemize}
\item[\rm (i)]
$A,A^*$ has spin;
\item[\rm (ii)]
$A,A^*$ is self-dual
and ${\mathcal Z}(A,A^*) \neq 0$.
\end{itemize}
\end{theorem}

\begin{proof}
(i) $\Rightarrow$ (ii)
By Lemma \ref{lem:SDpre}, the Leonard pair $A,A^*$ is self-dual.
By this and Lemma \ref{lem:sdLP}, there exists a self-dual Leonard system $\Phi$
associated with $A,A^*$.
By Section \ref{sec:main},
 Lemma \ref{lem:typesd},
and  Proposition \ref{prop:spinLP},
we get ${\mathcal Z}(\Phi)  \neq 0$.
By this and Definition \ref{def:ZLP}, ${\mathcal Z}(A,A^*) \neq 0$.

(ii) $\Rightarrow$ (i)
By Lemma \ref{lem:sdLP}
there exists a self-dual Leonard system  $\Phi$ associated with $A,A^*$.
By  ${\mathcal Z}(A,A^*) \neq 0$ and Definition \ref{def:ZLP}, ${\mathcal Z}(\Phi) \neq 0$.
By this and Section \ref{sec:main}, Lemma \ref{lem:typesd},
and Proposition \ref{prop:spinLP},
we find that $A,A^*$ has spin.
\end{proof}

\begin{corollary}    \label{cor:cheracterize}   \samepage
\ifDRAFT {\rm cor:characterize}. \fi
Let
$\Phi = ( A; \{E_i\}_{i=0}^d; A^*; \{E^*_i\}_{i=0}^d)$
denote  a self-dual Leonard system on $V$.
Then the following {\rm (i)--(iii)}
 are equivalent:
\begin{itemize}
\item[\rm (i)]
The Leonard pair $A,A^*$ has spin;
\item[\rm (ii)]
there exist scalars $f_0,f_1,f_2,f_3$ (not all zero) such that
\[
f_0 +f_1 \th^*_i + f_2 a_i + f_3 a_i \th^*_i = 0
\qquad\qquad
(0 \leq i \leq d);
\]
\item[\rm (iii)]
for $0 \leq i,j \leq d$,
\[
 (a_i - a_0)(\th^*_i - \th^*_d)(a_j - a_d)(\th^*_j - \th^*_0)
= (a_i - a_d)(\th^*_i - \th^*_0)(a_j - a_0)(\th^*_j - \th^*_d).
\]
\end{itemize}
\end{corollary}

\begin{proof}
By 
Definition \ref{def:ZLP}, 
Lemma \ref{lem:cond2},
Proposition \ref{prop:cond}, 
 and Theorem \ref{thm:characterize}.
\end{proof}

\begin{note}    \label{note:W}   \samepage
\ifDRAFT {\rm note:W}. \fi
For each spin Leonard pair, the corresponding $W, W^*$ from Definition \ref{def:spinLP}
are given in \cite[Lemma 1.17 and Theorem 1.18]{Cur}.
\end{note}

We finish this paper with some comments.
Recall our assumption $d \geq 3$.
It turns out that Corollary \ref{cor:cheracterize}  is false for $d=2$.
To see this, note that under the assumption $d=2$
the conditions (ii) and (iii) of Corollary \ref{cor:cheracterize} are vacuously true.
However under the assumption $d=2$ 
the condition (i) of Corollary \ref{cor:cheracterize} might not be true,
as the following example shows.
For this example we assume that the field $\F$ has characteristic zero.
Consider the algebra $\text{Mat}_3 (\F)$  of $3 \times 3$ matrices
that have all entries in $\F$. We index the rows and columns by $0,1,2$.
Consider the following matrices in $\text{Mat}_3 (\F)$:
\[
A =
\begin{pmatrix}
 1 & 0 & 0 \\
1 & 2 & 0  \\
0 & 1 & 5
\end{pmatrix},
\qquad\qquad
A^* =
\begin{pmatrix}
 1 & -1 & 0 \\
0 & 2 & -9 \\
0 & 0 & 5
\end{pmatrix}.
\]
Each of $A,A^*$ is multiplicity-free with eigenvalues $1,2,5$.
Using \eqref{eq:Ei} we find that the corresponding primitive
idempotents are
\begin{align*}
E_0 &=
\begin{pmatrix}
  1 & 0 &  0 \\
 -1 & 0 & 0 \\
 1/4 & 0 & 0
\end{pmatrix},
&
E_1 &=
\begin{pmatrix}
  0 & 0 &  0 \\
 1 & 1 & 0 \\
 -1/3 & -1/3 & 0
\end{pmatrix},
&
E_2 &=
\begin{pmatrix}
  0 & 0 &  0 \\
 0 & 0 & 0 \\
 1/12 & 1/3 & 1
\end{pmatrix},
\\
E^*_0 &=
\begin{pmatrix}
  1 & 1 &  9/4 \\
 0 & 0 & 0 \\
 0 & 0 & 0
\end{pmatrix},
&
E^*_1 &=
\begin{pmatrix}
  0 & -1 &  -3 \\
 0 & 1 & 3 \\
 0 & 0 & 0
\end{pmatrix},
&
E^*_2 &=
\begin{pmatrix}
  0 & 0 &  3/4 \\
 0 & 0 & -3 \\
 0 & 0 & 1
\end{pmatrix}.
\end{align*}
By matrix multiplication, we obtain
\begin{align*}
 E^*_i A E^*_j =
 \begin{cases}
  0 & \text{if $|i-j|>1$}, \\
 \neq 0 & \text{if $|i-j|=1$}
 \end{cases}
\qquad\qquad (0 \leq i,j \leq 2),
\\
 E_i A^* E_j =
 \begin{cases}
  0 & \text{if $|i-j|>1$}, \\
 \neq 0 & \text{if $|i-j|=1$}
 \end{cases}
\qquad\qquad (0 \leq i,j \leq 2).
\end{align*}
Therefore 
$\Phi = (A; \{E_i\}_{i=0}^2; A^* ; \{E^*_i\}_{i=0}^2)$
is a Leonard system.
By construction $\Phi$ is self-dual.
We are going to show that the Leonard pair $A,A^*$
does not have spin.
To do this, we assume that $A,A^*$ has spin and get a contradiction.
Consider the matrices $W$, $W^*$ from  Definition \ref{def:spinLP}.
We have  $W A = A W$ and  $W^* A^* = A^* W^*$ and
\begin{equation}
 W^* W A^* = A W^* W.     \label{eq:equ}
\end{equation}
The matrix $W$ commutes with $A$, and $A$ is multiplicity-free.
So $W$ is a polynomial in $A$.
Therefore $W \in \text{Span}\{E_0, E_1, E_2\}$.
Similarly
 $W^* \in \text{Span}\{E^*_0, E^*_1, E^*_2\}$.
There exist scalars $g_0, g_1, g_2$, $g^*_0, g^*_1, g^*_2$
such that
\[
  W  = \sum_{i=0}^2 g_i E_i,
\qquad\qquad
  W^* =  \sum_{i=0}^2 g^*_i E^*_i.
\]
Since $W$ and $W^*$ are invertible, we obtain
\[
g_i \neq 0,
\qquad\qquad
g^*_i \neq 0
\qquad \qquad
(0 \leq i \leq 2).
\]
In \eqref{eq:equ}, evaluate the $(2,2)$-entry of each side to get
\begin{equation}
 g_1 g^*_2 - g_2 g^*_1 = 0.                  \label{eq:22}
\end{equation}
In \eqref{eq:equ}, evaluate the  $(2,1)$-entry of each side to get
\begin{equation}
 -12 g_2 g^*_1 +(-3 g_0 + 4 g_1 - g_2) g^*_2 = 0.    \label{eq:21}
\end{equation}
In \eqref{eq:equ}, evaluate the  $(2,0)$-entry of each side to get
\begin{equation}
3(g_0 - g_2)  g^*_1 + (-3 g_0 + 4 g_1 - g_2) g^*_2 = 0.    \label{eq:20}
\end{equation}
In \eqref{eq:equ}, evaluate the  $(0,1)$-entry of each side to get
\begin{equation}
9(-g_0 + g_2) g^*_0 -4 ( g_0 + 3 g_2) g^*_1 + 3(- g_0 +   g_2) g^*_2 = 0.     \label{eq:01}
\end{equation}
In \eqref{eq:equ}, evaluate the  $(1,0)$-entry of each side to get
\begin{equation}
(-9g_0 -4g_1-3 g_2) g^*_0
+ 3(3 g_0 - 4 g_1 +  g_2) g^*_2 = 0.                     \label{eq:10}
\end{equation}
Combining \eqref{eq:22}--\eqref{eq:10} we get
$g_0 g^*_0 = 0$,
a contradiction.
Therefore the Leonard pair $A, A^*$  does not have spin.

\newpage

Kazumasa Nomura

Institute of Science Tokyo

Kohnodai, Ichikawa, 272-0827 Japan

email: knomura@pop11.odn.ne.jp

\medskip
Paul Terwilliger

Department of Mathematics

University of Wisconsin

480 Lincoln Drive

Madison, Wisconsin, 53706 USA

email: terwilli@math.wisc.edu

\bigskip

Keywords. Leonard pair, Leonard system, spin Leonard pair

2020 Mathematics Subject Classification. 05E30, 15B10.

\section{Statements and Declarations}

\noindent
{\bf Funding:} The authors declare that no funds, grants, or other support were received during
the preparation of this manuscript.

\medskip
\noindent
{\bf Competing interests:} The authors have no relevant financial or non-financial interests to
disclose.

\medskip
\noindent
{\bf Data availability:} All data generated or analyzed during this study are included in this
published article.


\begin{thebibliography}{10}

\bibitem{CaW}
J. S. Caughman IV,
N. Wolff,
The Terwilliger algebra of a distance-regular graph that
supports a spin model,
J. Algebraic Combin.\ 21 (2005) 289--310.


\bibitem{Cur}
B. Curtin,
Spin Leonard pairs,
Ramanujan J. 13 (2007) 319--332.

\bibitem{Cur2}
B. Curtin,
Modular Leonard triples,
Linear Algebra Appl.\
424 (2007) 510--539.

\bibitem{NT:affine}
K. Nomura, P. Terwilliger,
Affine transformations of a Leonard pair,
Electron. J. Linear Algebra 16 (2007) 389--417;
{\tt arXiv:math/0611783}.

\bibitem{NT:trid}
K. Nomura, P. Terwilliger,
Linear transformations that are tridiagonal with respect
to both eigenbases of a Leonard pair,
Linear Algebra Appl.\ 420 (2007) 198--207;
{\tt arXiv:math/0605316}.

\bibitem{NT:sd}
K. Nomura, P. Terwilliger,
Self-dual Leonard pairs,
Special Matrices
7 (2019) 1--19;
{\tt arXiv:1805.02545}.

\bibitem{NT:spin}
K. Nomura, P. Terwilliger,
Leonard pairs, spin models, and distance-regular graphs,
J. Combin. Theory (A) 177 (2021) 105312;
{\tt arXiv:1907.03900}.

\bibitem{NT:Nbip}
K. Nomura, P. Terwilliger,
Near-bipartite Leonard pairs,
Electron.\  J. Linear Algebra 40 (2024) 224--273;
{\tt arXiv:2304.04965}.

\bibitem{NT:spin2}
K. Nomura, P. Terwilliger,
Spin models and distance-regular graphs of $q$-Racah type,
European J. Combin. 124 (2025) Paper No,\ $104069$, $37$ pp.,
{\tt arXiv:2308.11061}.





\bibitem{T:Leonard}
P. Terwilliger,
Two linear transformations each tridiagonal with respect to an eigenbasis of
the other,
Linear Algebra Appl.\ 330 (2001) 149--203;
{\tt arXiv:math/0406555}.


\bibitem{T:TDD}
P. Terwilliger,
Two linear transformations each tridiagonal with respect to an eigenbasis of
the other; the TD-D canonical form
and the LB-UB canonical form,
J. Algebra, 291 (2005) 1--45;
{\tt arXiv:math/0304077}.

\bibitem{T:parray}
P. Terwilliger,
Two linear transformations each tridiagonal with respect to an eigenbasis of
the other; comments on the parameter array,
Des.\ Codes Cryptogr.\ 34 (2005) 307--332;
{\tt arXiv:math/0306291}.




\bibitem{T:survey}
P. Terwilliger,
 Two linear transformations each tridiagonal with respect to an eigenbasis of the other; 
an algebraic approach to the Askey scheme of orthogonal polynomials.,
Lecture Notes in Math.\ 1883, Springer-Verlag,  Berlin  2006,   255--330;
{\tt arXiv:math/0408390v3}.

\bibitem{T:notes}
P. Terwilliger,
Notes on the Leonard system classification,
Graphs Combin.\ 37 (2021) 1687--1748;
{\tt arXiv:2003.09668}.


\bibitem{TV}
P. Terwilliger, R. Vidunas,
Leonard pairs and the Askey-Wilson relations,
J. Algebra Appl.\ 3 (2004) 411--426.

\end{thebibliography}
\end{document}